\documentclass[11pt]{article}

\usepackage[T1]{fontenc}
\usepackage[utf8]{inputenc}

\usepackage[a4paper, margin=2cm]{geometry}
\usepackage{amsmath}
\usepackage{amssymb}
\usepackage{amsthm}
\usepackage{bbm}
\usepackage[title]{appendix}
\usepackage{xcolor}
\usepackage{hyperref}

\newcommand{\dd}{\mathrm{d}}
\newcommand{\ddiv}{\mathrm{div\,}}

\newcommand{\D}{\mathbb{D}}
\newcommand{\T}{\mathbb{T}}
\newcommand{\R}{\mathbb{R}}

\newtheorem{thm}{Theorem}
\newtheorem{lem}[thm]{Lemma}
\newtheorem{prop}[thm]{Proposition}
\newtheorem{rem}[thm]{Remark}

\newtheorem{de}[thm]{Definition}

\numberwithin{equation}{section}
\numberwithin{thm}{section}

\title{Discontinuous shear-thickening asymptotic \\ for power-law systems related to compressible flows
 }

\author{Didier Bresch$^\dagger$, Cosmin Burtea$^\S$, Maja Szlenk$^{\dagger,*}$}

\begin{document} 
\maketitle

{
\footnotesize

\centerline{$^\dagger\;$LAMA, UMR CNRS 5127, Universit\'e Savoie Mont Blanc, }
\centerline{B\^{a}t. Le Chablais, Campus Scientifique 73376 Le Bourget du Lac, France}

\bigbreak
\centerline{$^\S\;$ Université Paris Cité and Sorbonne Université, CNRS, IMJ-PRG, F-75013 Paris, France}
\centerline{Batiment Sophie Germain, Bureau 727, 8 place Aurélie Nemours,
75013 Paris;}

\bigbreak
\centerline{$^*\;$Institute of Applied Mathematics and Mechanics, University of Warsaw, }
\centerline{ul. Banacha 2, Warsaw 02-097, Poland}
}

\begin{abstract} In this paper we study the convergence of a power-law model for dilatant compressible fluids to a class of models exhibiting a maximum admissible shear rate, called thick compressible fluids. These kinds of problems were studied previously for elliptic equations, stating with the work of Bhattacharya, E. DiBenedetto and J. Manfredi [Rend. Sem. Mat. Univ. Politec. Torino 1989], and more recently for incompressible fluids by J.F. Rodrigues [J. Math. Sciences 2015]. Our result may be seen as an extension to the compressible setting of these previous works. Physically, this is motivated by the fact that the pressures generated during a squeezing flow are often large, potentially requiring the consideration of compressibility, see M. Fang and R. Gilbert [Z. Anal. Anwend 2004].


 Mathematically, the main difficulty in the compressible setting concerns the strong hyperbolic-parabolic coupling between the density and velocity field.  We obtain two main results, the first concerning the one-dimensional non-stationary compressible power-law system while the second one concerns the semi-stationary multi-dimensional case. Finally, we present an extension in one-dimension for a viscous Cauchy stress with singular dependence on the shear rate.
\end{abstract}

\date{\today}

\maketitle

\section{Introduction and main results}

Discontinuous shear thickening is an intriguing phenomenon that can be observed in non-Newtonian fluids.  More
precisely, this effect can be observed for instance in densely packed suspensions and colloids such as cornstarch in water (volume fraction greater than $0,55-0,58$
depending on the source): at shear-rates greater than a certain critical value there is a reversible step-like viscosity increase covering several
powers of ten in magnitude, see for instance
\cite{Hoffman1972,LaunBungSchmidt1991}.

Understanding this type of behaviour
is important for industrial applications in order to prevent overloading and
jamming \cite{Barnes1989, LaunBungSchmidt1991, PrabhuAjeeth2021} but also in damping devices, armour industry or sports equipement design
\cite{PrabhuAjeeth2021}. This can be regarded as a shear-rate-induced phase
transition from a low viscosity to a high viscosity solid-like state. 

To quantify the stress-strain relationship in the shear thickening regions, in
\cite{BrownJaeger2009} the viscosity shear stress $\tau$ is measured for different
values of shear-rate $\dot{\gamma}$: one observes a sudden increase at a
critical value $\dot{\gamma}_{c}$. Near this critical value the stress-strain
curves are locally fitted to a power law $\tau\approx\dot{\gamma}^{\frac
{1}{\varepsilon}}$ where the parameter $\varepsilon$ depends on packing
fraction and is small. As explained in \cite{BrownJaeger2009}, Newtonian flow
corresponds to $\varepsilon=1$ and a stress divergence corresponds to
$\varepsilon=0$.

 Note that it is well known that gradient or strain-tensor type constraints arise also in the mathematical formulation of several problems in mechanics and in physics, namely in critical state models of plasticity, superconductivity and also in geophysics descriptions (for instance concerning sandpile growth or formation of network of lakes and rivers) leading to variational or quasi-variational inequalities with behind unilateral constraints. In addition, as indicated in \cite{FaGi}, the pressures generated during a squeezing flow are often large and give rise to the possible necessity of taking compressibility into consideration for instance for such situations in the deep oceans.

 \medskip
 
 Motivated by all these considerations, we consider compressible viscous flows for which the density $\varrho_p$ and the velocity $u_p$ are related by the mass equation
 \begin{equation}\partial_t \varrho_p 
 + {\rm div}(\varrho_p u_p) = 0,\label{mass_intro}\end{equation}
 coupled with the momentum equation
 \begin{equation}\partial_t(\varrho_p u_p)
+ {\rm div} (\varrho_p u_p\otimes u_p)
-\ddiv(|\D u_p|^{p-2}\D u_p) + a\nabla\varrho_p^\gamma = 0
\label{momentum_intro}
\end{equation}
where  $\D u = (\nabla u+ (\nabla u)^t) /2$.

 In this paper, we are interested in a mathematical justification of a unilateral constraint on the modulus of strain tensor $\D u$.
More precisely, letting $p$ tend to infinity we prove that we get the unilateral constraint 
\[ |\D u|\le 1, \qquad
   \pi \ge 0 \quad \hbox{with} \quad \pi (1-|\D u|)= 0 \]
where $\pi$ is the Lagrangian multiplier associated to the maximal strain constraint $|\D u| =1.$
The readers interested by such unilateral constraint, related to the $p$-laplacian or to the $p$ system
related to elasticity, may consult the important papers \cite{bhattacharya_et_al1989,BouButDeP, DPEvPr,Evans}. See also \cite{ArEv} and \cite{LeeFolTuck} for discussions on the compression molding for polymers with a link to the limit $p\to\infty$.
One can also consult the interesting paper related to the numerical approximation of the $\infty$-harmonic mapping \cite{KatzPryer}. Concerning papers related to discontinuous shear thickening fluid, we refer to \cite{DLR-St} for the stationary incompressible system. 

Our goal is to extend the results previously obtained for the limit $p\rightarrow \infty$ for the $p$-laplacian or models for incompressible power-law fluids (see \cite{de2014,miranda2016variational,Rodrigues,sazhenkov1998}) to the compressible setting. 
In this case, the situation becomes tricky due to the mixed hyperbolic-parabolic nature of the equations.
We are able to treat two cases~: 
\begin{itemize}
\item The non-stationary case in a one space periodic dimension $\Omega= \T$   where we obtain a PDE for the limit $p\rightarrow\infty$ in the spirit of De Pascale, Evans, Prateli \cite{DPEvPr}.
\item The semi-stationary Stokes system in the multidimensional periodic domain $\Omega= \T^d$ that means the case \eqref{mass_intro}-\eqref{momentum_intro}: The total time derivative term $\partial_t(\varrho_pu_p) + {\rm div}(\varrho_p u_p \otimes u_p)$ in \eqref{momentum_intro} is neglected, while we obtain a limit in the framework of variational solutions. 
\end{itemize}

Since our starting point is the power-law system for a fixed $p$ (before letting $p\to\infty$), let us mention the existing results concerning this case. In the incompressible setting, the results concerning weak solutions start from Ladyzhenskaya \cite{Ladyzhenskaya1967}, and were thoroughly developed later on for different values of $p$. See for instance \cite{DiRuWo,FrMaSt,MaNeRu,wolf}. For the incompressible non-homogenous case, let us mention for example \cite{FrMaRu}. Concerinig the compressible system, the steady case was considered in \cite{BurSzl}. In the context of our work, another interesting result is obtained in \cite{FLM} for viscous stress tensors with a singular dependence with respect to the shear-rate which ensures that the divergence of the velocity field is bounded. We will discuss such models at the end of the paper.

Let us now explain in more detail the systems we consider, and the corresponding limits we obtain when $p\to\infty$. The paper is separated into two parts. First, we consider the non-stationary compressible system in 1D, and then the transport-Stokes system in multiple dimensions. In the last section we finish with an extension in one-dimensional case replacing the power-law stress law by a singular 
stress law.

\bigskip
\noindent {\bf Non-stationary compressible system in 1D and its $p \to \infty$ corresponding model.}
The first case we consider is the following $1$D power-law compressible viscous flows system:%
\begin{equation}
\left\{
\begin{aligned}
&\partial_{t}\varrho_p +\partial_{x}\left(\varrho_p u_p\right) =0,\\
&\partial_{t}\left(\varrho_p u_p\right)  +\partial_{x}\left(\varrho_p
u_p^{2}\right)  -\mu\partial_{x}(\left\vert \partial_{x}u_p\right\vert
^{p-2}\partial_{x}u_p)+a\partial_{x}\varrho_p^{\gamma} =0,
\end{aligned}
\right.\label{p-system}
\end{equation}
in a periodic setting in space and satisfying
\begin{equation}\label{p-systemBC}
\varrho_p\vert_{t=0}= \varrho_{p,0}, \qquad 
  \varrho_p u_p\vert_{t=0}= \varrho_{p,0} u_{p,0}.
\end{equation}
Additionally, the solution satisfies
\begin{equation}
\label{p-system-mass}
\int_{\T}\varrho_{p}\mathrm{d}x= \int_{\T} \varrho_{p,0} dx = M_{p,0}
\qquad \text{ and } \qquad \int_{\T} \varrho_p u_{p}\mathrm{d}x= \int_{\T} \varrho_{p,0}u_{p,0} dx = \widetilde M_{p,0}. 
\end{equation}
The main theorem regarding the system 
 \eqref{p-system}--\eqref{p-system-mass} concerns global existence of a solution with weak regularity and the asymptotic limit $p\to +\infty$, namely 

\begin{thm}\label{MainTh1} 
For $p\geq2$, let us consider $(\varrho_{p,0},u_{p,0})$ and $(\varrho_{0},u_{0})\in
L^{\infty}(\T)\times W^{1,\infty}(\T)$ such that:
\[
\left\{
\begin{array}
[c]{l}%
\displaystyle 
\overline{E_{0}}:=\sup_{p\geq2}\int_{\T}\left(  \frac{\varrho_{p,0}u_{p,0}^{2}%
}{2}+\frac{\varrho_{p,0}^{\gamma}}{\gamma-1}\right)  \mathrm{d}x<\infty.\\
\\
\displaystyle 
0<c_{1}\leq\varrho_{p,0}\leq c_{2}<+\infty 
\text{ and }\left\vert \partial_{x}u_{p,0}\right\vert \leq1,\\
\\
\displaystyle 
\varrho_{p,0}\to\varrho_{0},\text{ }u_{p,0}\to
u_{0}\text{ strongly in }L^{\infty}
\end{array}
\right.
\]
with $\gamma>1$ for some finite constants $c_{1},c_{2}>0$ independent of $p$. Then, for fixed $p$ large enough, there exists a global
weak solution of \eqref{p-system}--\eqref{p-system-mass} such that the energy
inequality holds, i. e.
\[
\sup_{t\in(0,T]}\int_{\T}\left[\frac{1}{2}\varrho_{p}|u_{p}|^{2}+\frac{1}%
{\gamma-1}\varrho_{p}^{\gamma}\right](t)\mathrm{d}x+\int_{0}^{t}\int_{\T}\left\vert \partial_{x}u_{p}\right\vert ^{p}\;\mathrm{d}x\leq\overline
{E_{0}},
\]
and
\begin{gather*}
0<C_{1}\left(  t\right)  \leq\varrho_{p}\left(  t,x\right)  \leq C_{2}\left(
t\right)<+\infty  ,\\
\sup_{0<t<T}\frac{1}{p}\Vert\partial_{x}u_{p}\left(  t\right)  \Vert
_{L^{p}(\T)}^{p}+\left\Vert \partial_{t}u_{p}\right\Vert _{L^{2}(\left(
0,T\right)  \times  \T  )}\leq 
C_{3} \left(  T\right)
<+\infty,
\end{gather*}
with $C_{i}\left(  T\right)  $ $i\in\overline{1,3}$ functions depending on
$\overline{E_{0}},c_{1},c_2$ and on $T>0$ but independent on $p$. 
Moreover, assuming $p\to +\infty$, we have
\begin{equation}
\label{convp}
\begin{aligned}
u_{p} &\rightharpoonup u \hbox{ weakly in }{\mathcal{C}}(0,T;L^{2}(\T))\cap L^{r}(0,T,W^{1,r}(\T))\hbox{ for all }r<+\infty,\\
u_{p} &\rightarrow u \hbox{ strongly in }\text{ }L^{2}\left(  0,T;L^{2}%
(\T)\right)  ,\\
\partial_{t}u_{p} &\rightharpoonup\partial_{t}u \hbox{ weakly-* in }L^{2}%
(0,T;L^{\infty}(\T)),\\
\tau_{p}=|\partial_{x}u_{p}|^{p-2}\partial_{x}u_{p} &\rightharpoonup
\tau \hbox{ weakly in }L^{2}(0,T;L^{2}(\T)),\\
\varrho_{p} &\rightarrow\varrho \hbox{ strongly in }{\mathcal{C}}([0,T];L^{p}%
(\T))\hbox{ for all }p<+\infty
\end{aligned}
\end{equation}
where 
$$ (\varrho,u) \in
L^\infty\big((0,T)\times(\T)\big) 
\times
\Bigl(H^1(0,T;L^2(\T))\cap L^\infty(0,T;W^{1,r}(\T))
\Bigr) 
\hbox{ for any } r<+\infty$$ is a
solution of the variational inequality
$$\int_0^T\int_{\T} \varrho  \, 
\partial_t u \, (v-u)  
+ \varrho u \partial_x u  (v-u) \;\dd x\dd t
- \int_0^T\int_{\T} \varrho^\gamma  
\partial_x (v-u) \;\dd x\dd t
\ge 0 $$
for all $v\in C_{c}^\infty([0,T]\times\T)$ such that $|\partial_x v|\le 1$,
 with
$$
(\varrho, u)_{|_{t=0}}=(\varrho_0,u_0)
$$
where $\varrho$ satisfies
$$\partial_t \varrho + \partial_x(\varrho u)=0 
\hbox{ in a weak sense.} $$

\end{thm}

\begin{rem} It is important to remark that \eqref{p-system-mass} contributes to get a bound in time on $\displaystyle\widetilde u_p= \int_{\T} u_p \;\dd x$ in the periodic setting even if $\varrho_p$ vanishes or blows-up due to the total time derivative in the momentum equation. Indeed, we just have to write
$$M_{p,0} \widetilde u_p =  \int_{\T} \varrho_p (\widetilde u_p - u_p)dx
+ \int_{\T} \varrho_p u_p dx.$$
The second quantity is equal to $\widetilde M_{p,0}$ and we can
bound the first quantity using the bounds on $\varrho_p$
and on $\partial_x u_p$ (through Poincar\'e--Wirtinger inequality) given by the energy estimates.
\end{rem}

\begin{rem}
The condition $\left\vert \partial_{x}u_{p,0}\right\vert \leq1$ is needed in
order to obtain estimates uniformly with respect to $p$ but it is not
necessary in order to probe existence at fixed $p$. 
\end{rem}

\begin{rem}
Note that since for fixed $p\geq2$ the density is bounded and far from vacuum, the same is true in the limit, i. e.
$$ 0 < c_1(t) \le \varrho (t,x) \le C_1(t) <+\infty$$
for all $t>0$.
\end{rem}
Using the higher regularity of solutions for a fixed $p$, we also derive the stronger version of Theorem \ref{MainTh1}:

\begin{thm} \label{MainTh12} 
Assume the initial data satisfy the same assumptions as in Theorem {\rm \ref{MainTh1}}. Then, there exists a subsequence converging to $(\varrho,u)$ which satisfies the system
\begin{equation}
\left\{
\begin{array}
[c]{r}%
\partial_{t}\varrho +\partial_{x}\left(  \varrho u\right)  =0,\\
\partial_{t}\left(  \varrho u\right)  +\partial_{x}\left(  \varrho
u^{2}\right)  -\mu\partial_{x}\tau + a \partial_{x}\varrho^{\gamma}=0, \\
|\partial_x u |\le  1 \qquad  \hbox{ and } \qquad  \tau = \pi \partial_x u \hbox{ with }
\pi \ge 0 
\hbox{ with } 
\pi (1-|\partial_x u |) = 0,  \\
\varrho\vert_{t=0} = \varrho_0, \qquad
\varrho u\vert_{t=0}= \varrho_0 u_0
\end{array}
\right.\label{limit-system}
\end{equation}

\end{thm}

\begin{rem} Theorem {\rm \ref{MainTh12}}  provides an unilateral constraint on  the strain tensor. It is important to note that $\pi$ may be seen as the Lagrange multiplier problem related to the maximal constraint $|\partial_x u|=1$. Note also that the formulation
\eqref{limit-system} with the regularity found on $(\varrho,u)$ given in Theorem {\rm \ref{MainTh1}} is equivalent to the problem
\end{rem}

\begin{rem} Remark that recently a unilateral constraint related to the maximal packing for a compressible Navier-Stokes  system have been obtained as a singular limit ($\gamma$ tends to infinity) of the usual compressible Navier-Stokes system
$$\partial_t \varrho_\gamma + {\rm div} (\varrho_\gamma u_\gamma) = 0$$
$$\partial_t (\varrho_\gamma u) + {\rm div} (\varrho_\gamma u_\gamma\otimes u_\gamma)
- {\rm div} \sigma_\gamma = 0.$$ 
 where the
stress tensor is given by 
$$\sigma_\gamma = 2\mu \Delta u_\gamma 
+ \lambda \,  {\rm div} u_\gamma {\rm Id}
- a (\varrho_\gamma)^\gamma {\rm Id}$$
with $\mu$ and $\lambda$ positive constants such that
$\lambda + 2\mu/d>0$ and $a$ is a strictly positive constant. The unilateral constraint reads
$$0\le \varrho \le 1, \qquad 
   \pi\ge 0 \hbox{ with } \pi (1-\varrho)= 0$$
where $\pi$ is the Lagrangian multiplier associated to the maximal packing constraint $\varrho =1.$
The interested reader is referred to \cite{LionsMasmoudi} for a mathematical justification considering global weak solutions \`a la Leray. Remark that there exists another way to get Navier-Stokes equations with the unilateral constraint which is to start with a singular pressure $p(\varrho) = \varepsilon \varrho^\gamma/(1-\varrho)^\beta$ for some given $\gamma$ and $\beta$ and let $\varepsilon$ go to zero. See the papers \cite{BPZ,PerZat} for a mathematical justification related to weak solutions
à la Leray. See also \cite{LLMa} where they propose this macroscopic limit system with unilateral constraint starting from studies to understand contact between a grain with a wall.  \end{rem} 

\bigskip

\noindent {\bf Semi-stationary model in multiple dimensions and its 
$p\to +\infty$ system.} In the second part of the paper, we consider the multi-dimensional semi-stationary problem, namely the following compressible Stokes system for a power-law fluid:
\begin{equation}\label{eq_p}
\left\{
    \begin{aligned}
        \partial_t\varrho + \ddiv(\varrho u) &= 0, \\
         - \ddiv\left(|\D u|^{p-2}\D u\right) + \nabla \varrho^\gamma &= 0
    \end{aligned}
    \quad \text{in} \quad [0,T]\times\T^d
    \right.
\end{equation}
with the initial condition 
\begin{equation}\label{BC}
  \varrho\vert_{t=0} = \varrho_0.  
\end{equation}
We will determine the limit system as $p\to\infty$ and show that having a sequence $(\varrho_p,u_p)$ of solutions to (\ref{eq_p}), there exists a subsequence converging to the limit system. However, due to the lower regularity of the solution, we are not able to derive the multidimensional equlivalent of (\ref{limit-system}). Instead, we work in the framework of variational solutions. Testing (\ref{eq_p}) by $u_p$, we obtain the usual energy estimate
\begin{equation}\label{energy_multid} \frac{1}{\gamma-1}\sup_{t\in [0,T]}\int_{\T^d}\varrho_p^\gamma\;\dd x + \int_0^T\int_{\T^d} |\D u_p|^p\;\dd x \leq \frac{1}{\gamma-1}\int_{\T^d}\varrho_0^\gamma\;\dd x. \end{equation}
Moreover, testing the momentum equation by $u_p-v$ and using the relation 
\[ \frac{1}{p}|\D v|^p \geq \frac{1}{p}|\D u|^p + |\D u|^{p-2}\D u:(\D u-\D v), \]
we obtain the variational formulation for (\ref{eq_p}):
\begin{equation}\label{variational} \frac{1}{p}\int_0^t\int_{\T^d} |\D u_p|^p\;\dd x\dd s - \int_0^t\int_{\T^d} \varrho_p^\gamma(\ddiv u_p-\ddiv v)\;\dd x\dd s \leq \frac{1}{p}\int_0^t\int_{\T^d} |\D v|^p\;\dd x\dd s \quad \text{for a. a. } t\in [0,T] \end{equation}
for all $v\in C^\infty_0([0,T]\times\T^d)$. This brings a motivation for the definition of the solution.

\begin{de}
We say that $(\varrho_p,u_p)$ is a variational solution to (\ref{eq_p}), if it satisfies
\begin{enumerate}
\item the energy estimate (\ref{energy_multid}),
\item the continuity equation in the distributional sense: for each $\varphi\in C_0^\infty([0,T]\times\T^d;\R)$
\[ -\int_0^T\int_{\T^d} \varrho_p\partial_t\varphi\;\dd x\dd t - \int_0^T\int_{\T^d}\varrho_pu_p\cdot\nabla\varphi\;\dd x\dd t = 0, \]
\item the momentum equation in the variational sense: for each $v\in C_0^\infty([0,T]\times\T^d;\R^d)$
\[ \frac{1}{p}\int_0^t\int_{\T^d} |\D u_p|^p\;\dd x\dd s - \int_0^t\int_{\T^d} \varrho_p^\gamma(\ddiv u_p-\ddiv v)\;\dd x\dd s \leq \frac{1}{p}\int_0^t\int_{\T^d} |\D v|^p\;\dd x\dd s \quad \text{for a. a. } t\in [0,T] \]
\end{enumerate}
\end{de}

Now, assuming we have sufficient compactness properties of $(\varrho_p,u_p)$, after taking $p\to\infty$ we expect the limit to satisfy the following inequality for a. a. $t\in[0,T]$:
\begin{equation}\label{eq_limit}
    -\int_0^t\int_{\T^d} \varrho^\gamma(\ddiv u-\ddiv v)\;\dd x\dd s \leq I[\D v\mathbbm{1}_{[0,t]}] \quad \text{for all} \quad v\in C_0^\infty([0,T]\times\T^d), 
\end{equation}
where
\[ I[\varphi] = \left\{\begin{aligned}
    &0, \quad \text{if} \quad \|\varphi\|_{L^\infty([0,+\infty)\times\T^d)} \leq 1, \\
    &+\infty \quad \text{otherwise}.
\end{aligned}\right.\]

Therefore our main result for the multidimensional case states as follows:

\begin{thm}\label{th_multid}
    Assume that $(\varrho_p,u_p)$ is a weak solution to \eqref{eq_p}-\eqref{BC}  with the initial data $\varrho_0\in L^\infty(\T^d)$, satisfying the energy bound
    \begin{equation}\label{energy} \frac{1}{\gamma-1}\sup_{t\in[0,T]}\int_{\T^d}\varrho_p^\gamma\;\dd x + \int_0^T\int_{\T^d} |\D u_p|^p\;\dd x\dd t \leq \frac{1}{\gamma-1}\int_{\T^d} \varrho_0^\gamma. \end{equation}
    Then, there exist $(\varrho,u)$ 
    satisfying
    \[ \varrho\in L^\infty([0,T]\times\T^d), \quad u\in L^q(0,T;W^{1,q}(\T^d)) \quad \text{for each} \quad q<\infty, \quad \text{and} \quad |\D u|\leq 1 \quad \text{a. e.} \]
    such that up to a subsequence we have the convergence
    \[\begin{aligned}
        \varrho_p \to \varrho \quad &\text{strongly in} \quad L^\gamma(0,T;L^\gamma(\T^d)), \\
        u_p \rightharpoonup u \quad &\text{weakly in} \quad L^d(0,T;W^{1,d}(\T^d)).
    \end{aligned}\]
    Moreover, $(\varrho,u)$ satisfies the variational inequality (\ref{eq_limit}) and the continuity equation 
    \[ \partial_t\varrho + \ddiv(\varrho u) = 0, \quad \varrho_{|_{t=0}} = \varrho_0 \]
    in the sense of distributions.
\end{thm}

Note that the construction of weak solutions to (\ref{eq_p}) remains an open problem, and therefore Theorem \ref{th_multid} is not constructive. However, by including an additional regularizing term for $\ddiv u$, using the same compactness method as in Theorem \ref{th_multid} we are able to construct the solution directly to the limit problem (\ref{eq_limit}). We therefore have
\begin{thm}\label{th_multid_ex}
    Let $\varrho_0\in L^\infty(\T^d)$. There exist $(\varrho,u)$ such that 
    \begin{enumerate}
        \item $\varrho\in L^\infty(0,\infty;L^\gamma)$ and $\varrho\in L^\infty([0,T]\times\T^d)$ for each $T>0$, 
        \item $u\in L^r(0,\infty;W^{1,r})$ for each $r<\infty$ and $|\D u|\leq 1$ a. e.,
        \item $(\varrho,u)$ satisfy the continuity equation 
        \[ \partial_t\varrho + \ddiv(\varrho u) = 0 \]
        in the sense of distributions, and moreover
        \begin{equation}\label{eq_limit2}
    -\int_0^t\int_{\T^d} \varrho^\gamma(\ddiv u-\ddiv v)\;\dd x\dd s \leq I_t[v] \quad \text{for all} \quad v\in C_0^\infty([0,+\infty)\times\T^d), 
\end{equation}
where
\[ I_t[v] = \left\{\begin{aligned}
    &0, \quad \text{if} \quad \|\D v\|_{L^\infty([0,t]\times\T^d)} \leq 1, \\
    &+\infty \quad \text{otherwise}.
\end{aligned}\right.\]
    \end{enumerate}
\end{thm}

\section{Proof of Theorem \ref{MainTh1}}\label{1d_section} The local existence of strong solutions can be achieved using a classical Galerkin method or fixed point theorem applied to the lagrangian formulation. We skip the details and we refer the reader for example to \cite{KaMaNe}. The main points is for us to derive appropriate uniform bounds with respect to $p$ to define weak solutions \`a la Hoff for fixed $p$ and then pass to the limit with respect to $p$. 

\subsection{Uniform estimates} 

The first step is to show the following {\it a priori} estimates:
\begin{prop}Let $(\varrho_p,u_p)$ be a smooth solution of \eqref{p-system}. Then
\begin{equation}
\mu\left\vert \partial_{x}u_{p}\right\vert ^{p-2}\partial_{x}u_{p}
-a\varrho_{p}^{\gamma}\leq\mu\left\vert \partial_{x}u_{p,0}\right\vert
^{p-2}\partial_{x}u_{p,0}-a\varrho_{p,0}^{\gamma}\leq\mu,
\end{equation}
and
\begin{equation}  \label{estim_density}%
\frac{c_{1} e^{-2t}}{\max\Bigl\{1, \bigl(\frac{a}{\mu}\bigr)^{1/\gamma} c_2
\Bigr\} }
\leq\varrho_{p}\left(
t,x\right)  \leq c_2 \exp\left[\frac{\overline{E_0}}{\mu} + \left(\frac{2+\gamma}{\mu}\overline{E_0} + 1+\frac{1}{p}\right)t\right].
\end{equation}
\end{prop}

\smallskip

\begin{rem} The readers interested by other mathematical studies (in one space dimension for compressible flows)
where deriving a parabolic equation on the stress quantity $\sigma$  is important are referred to
{\rm \cite{Constantin_2020}} and the recent preprint {\rm \cite{BrBuGJLa}}
\end{rem}

\smallskip
\begin{proof}
In order to obtain the bounds on $\partial_x u_p$, and consequently $\varrho_p$, we first estimate the Cauchy stress.
\paragraph{Bounds for the Cauchy stress.}
We use the maximum principle to obtain the following bound on the Cauchy stress %
\[
\sigma_{p}:=\mu|\partial_xu|^{p-2}  \partial_{x}u_{p}  -a\varrho_{p}^{\gamma}.%
\]
We observe that
\begin{equation}
\partial_{t}\partial_{x}u_{p}+u_{p}\partial_{xx}^{2}u_{p}+(\partial_{x}%
u_{p})^{2}-\partial_{x}\left(  \frac{1}{\varrho_{p}}\partial_{x}(\mu\left\vert
\partial_{x}u_{p}\right\vert ^{p-2}\partial_{x}u_{p}-a\varrho_{p}^{\gamma
})\right)  =0.\label{dx_1}%
\end{equation}
Now, let $H\left(  s\right)  =\left\vert s\right\vert^{p-2}s$. Then 
\[ H^{\prime
}\left(  s\right)  =\left(  p-1\right)  \left\vert s\right\vert ^{p-2}\geq0 \]
and its inverse is given by
\[
H^{-1}\left(  s\right)  =\left\vert s\right\vert ^{p^{\prime}-2}s\text{ with
}(H^{-1})^{\prime}\left(  s\right)  =\left(  p^{\prime}-1\right)  \left\vert
s\right\vert ^{p^{\prime}-2}\geq0,
\]
where $p'=p/(p-1)$. Note
that $p^{\prime}-2= p/(p-1)-2 = - (p-2)/(p-1).$
Multiplying $\left(  \text{\ref{dx_1}}\right)  $ with $\mu H^{\prime}\left(
\partial_{x}u_{p}\right)  $ gives us
\begin{equation}
\partial_{t}(\mu H\left(  \partial_{x}u_{p}\right))  +u_{p}\partial_{x}(\mu
H\left(  \partial_{x}u_{p}\right))  +\mu H^{\prime}\left(  \partial_{x}%
u_{p}\right)  (\partial_{x}u_{p})^{2}-\mu H^{\prime}\left(  \partial_{x}%
u_{p}\right)  \partial_{x}\left(  \frac{1}{\varrho_{p}}\partial_{x}\sigma
_{p}\right)  =0.\label{dx_2}%
\end{equation}
We obtain from $\left(  \text{\ref{dx_2}}\right)  $ that
\begin{align*}
&  \partial_{t}\sigma_{p}+u_{p}\partial_{x}\sigma_{p}+\mu H^{\prime}\left(
\partial_{x}u_{p}\right)  (\partial_{x}u_{p})^{2}-\mu H^{\prime}\left(
\partial_{x}u_{p}\right)  \partial_{x}(1/\varrho_{p}\partial_{x}\sigma_{p})\\
&  =-a\left(  \partial_{t}\varrho_{p}^{\gamma}+u_p\partial_{x}\varrho
_{p}^{\gamma}\right)  =a\gamma\varrho_{p}^{\gamma}\partial_{x}u_{p}.
\end{align*}
This gives us%
\begin{align*}
\partial_{t}\sigma_{p}+u_{p}\partial_{x}\sigma_{p}-\mu H^{\prime}\left(
\partial_{x}u_{p}\right)  \partial_{x}(1/\varrho_{p}\partial_{x}\sigma_{p}) &
=a\gamma\varrho_{p}^{\gamma}\partial_{x}u_{p}-\left(  p-1\right)
\mu\left\vert \partial_{x}u_{p}\right\vert^{p}\\
&  =\gamma\left(  -\sigma_{p}+\mu\left\vert \partial_{x}u_{p}\right\vert
^{p-2}\partial_{x}u_{p}\right)  \partial_{x}u_{p}-\left(  p-1\right)
\mu\left\vert \partial_{x}u_{p}\right\vert ^{p}\\
&  =-\gamma\sigma_{p}\partial_{x}u_{p}+\mu\left(  1+\gamma-p\right)
\left\vert \partial_{x}u_{p}\right\vert ^{p}.
\end{align*}
Finally, we obtain that%
\[
\partial_{t}\sigma_{p}+u_{p}\partial_{x}\sigma_{p}-\mu H^{\prime}\left(
\partial_{x}u_{p}\right)  \partial_{x}(1/\varrho_{p}\partial_{x}\sigma
_{p})=-\gamma\sigma_{p}\partial_{x}u_{p}+\mu\left(  1+\gamma-p\right)
\left\vert \partial_{x}u_{p}\right\vert ^{p}.
\]
Note that for $p$ sufficiently large, i.e.%
\[
p\geq 1+\gamma,
\]
the second term of the right hand side of the previous inequality is negative. Furthermore, observe
that%
\begin{align*}
-\gamma\sigma_{p}\partial_{x}u_{p} &  =-\frac{\gamma}{\mu^{\frac{1}{p-1}}%
}\sigma_p H^{-1}\left(  \mu\left\vert \partial_{x}u_{p}\right\vert
^{p-2}\partial_{x}u_{p}\right)  \\
&  =-\frac{\gamma}{\mu^{\frac{1}{p-1}}}\left[  H^{-1}\left(  \sigma
_{p}+a\varrho_{p}^{\gamma}\right)  -H^{-1}\left(  \sigma_{p}\right)  \right]
\sigma_{p}- \frac{\gamma}{\mu^{\frac{1}{p-1}}}H^{-1}\left(  \sigma_{p}\right)  \sigma_{p}\\
&  =-\frac{\gamma}{\mu^{\frac{1}{p-1}}}\left[  H^{-1}\left(  \sigma
_{p}+a\varrho_{p}^{\gamma}\right)  -H^{-1}\left(  \sigma_{p}\right)  \right]
\sigma_{p}-\frac{\gamma}{\mu^{\frac{1}{p-1}}}\left\vert \sigma_{p}\right\vert ^{p^{\prime}}.
\end{align*}
Denoting by
\[
\tilde{\sigma}_{p}\left(  t\right)  =\max_{x\in\left[  0,1\right]  }\sigma
_{p}\left(  t,x\right)
\]
we obtain using the maximum principle that
\[
\frac{d\tilde{\sigma}_{p}}{dt}\left(  t\right)  +a\left(  t\right)
\tilde{\sigma}_{p}\left(  t\right)  \leq0
\]
with $a\left(  t\right)  $ some positive term. It follows that
\begin{equation}
\max_{x\in\left[  0,1\right]  }\left(  \mu\left\vert \partial_{x}%
u_{p}\right\vert ^{p-2}\partial_{x}u_{p}-a\varrho_{p}^{\gamma}\right)  \left(
t,x\right)  =\max_{x\in\left[  0,1\right]  }\sigma_{p}\left(  t,x\right)
\leq\max_{x\in\left[  0,1\right]  }\sigma_{p,0}\left(  x\right)  \leq
\mu. \label{bounded_flux}%
\end{equation}
We obtain that%
\[
\sup_{x\in\left[  0,1\right]  }\left\vert \partial_{x}u_{p}\right\vert
^{p-2}\partial_{x}u_{p}\leq\frac{a}{\mu}\varrho_{p}^{\gamma}+1,
\]
from which
\[
\partial_{x}u_{p}\left(  t,x\right)  \leq\left(  \frac{a}{\mu}\varrho
_{p}^{\gamma}\left(  t,x\right)  +1\right)  ^{\frac{1}{p-1}}.
\]
Let us remark that this inequality in fact gives an estimate for the positive
part of $\partial_{x}u_{p}\left(  t,x\right)  .$

\bigskip

\paragraph{Lower bound for the density.}
The upper-bound on $\partial_x u$ allows to obtain immediately a lower bound for $\varrho_{p}$ by using
Lagrangian coordinates :%
\[
\dot{X}_{t}\left(  x\right)  =u_{p}\left(  t,X_{t}\left(  x\right)  \right)
\]
we have that%
\begin{align*}
\partial_{t}\varrho_{p}\left(  t,X_{t}\left(  x\right)  \right)   &
=-\varrho_{p}\left(  t,X_{t}\left(  x\right)  \right)  \partial_{x}%
u_{p}\left(  t,X_{t}\left(  x\right)  \right)  \\
&  \geq-\varrho_{p}\left(  t,X_{t}\left(  x\right)  \right)  \left(  \frac
{a}{\mu}\varrho_{p}^{\gamma}\left(  t,X_{t}\left(  x\right)  \right)
+1\right)  ^{\frac{1}{p-1}}\\
&  \geq-\left(  \frac{a}{\mu}\right)  ^{\frac{1}{p-1}}\varrho_{p}%
^{1+\frac{\gamma}{p-1}}\left(  t,X_{t}\left(  x\right)  \right)  -\varrho
_{p}\left(  t,X_{t}\left(  x\right)  \right),
\end{align*}
where in the last line we used the inequality 
$\left(  x+y\right)^\alpha\leq\left(
x^\alpha+y^\alpha\right)$ for $\alpha<1$. It follows that%
\begin{align*}
&  \partial_{t}\varrho_{p}^{-\frac{\gamma}{p-1}}\left(  t,X_{t}\left(
x\right)  \right)  \\
&  \leq\frac{\gamma}{p-1}\left(  \frac{a}{\mu}\right)  ^{\frac{1}{p-1}}%
+\frac{\gamma}{p-1}\varrho_{p}^{-\frac{\gamma}{p-1}}\left(  t,X_{t}\left(
x\right)  \right)  .
\end{align*}

Denoting $Y_p(s)=
\varrho_p^{-\gamma/(p-1)}(s,X_s(x))$, we can write
$$\partial_s Y_p(s) 
- \frac{\gamma}{p-1} Y_p(s) \le \frac{\gamma}{(p-1)} \bigl(\frac{a}{\mu}\bigr)^{1/(p-1)}$$
and therefore
$$\partial_s (Y_p(s) \exp(- \gamma s / (p-1)))
\le\frac{\gamma}{(p-1)} \bigl(\frac{a}{\mu}\bigr)^{1/(p-1)}\exp(- \gamma s / (p-1)) $$
a,d therefore integrating from $0$ to $t$
$$Y_p(t) \exp(- \gamma 
t / (p-1)) - Y_p(0)
\le \bigl(\frac{a}{\mu}\bigr)^{1/(p-1)}
[1- \exp(- \gamma t / (p-1))]
$$
and therefore we get
\[
\varrho_{p}^{-\frac{\gamma}{p-1}}\left(t,X_t\left(x\right)  \right)
\leq \left(  \frac{a}{\mu}\right)^{\frac{1}{p-1}}
\Bigl(\exp\left(
\frac{\gamma}{p-1}t\right)-1\Bigr)
+
\varrho_{p,0}^{-\frac{\gamma}{p-1}}\left(x\right)  \exp\left(
\frac{\gamma}{p-1}t\right)
\]
and therefore
\[ \varrho_p(t,X_t(x)) \geq \frac{1}{\left(\varrho_0^{-\frac{\gamma}{p-1}}e^{\frac{\gamma}{p-1}t} + \left(\frac{a}{\mu}\right)^{\frac{1}{p-1}}\left(e^{\frac{\gamma}{p-1}t}-1\right)\right)^{\frac{p-1}{\gamma}}} = \frac{\varrho_0e^{-t}}{\left(1+\left(\frac{a}{\mu}\varrho_0^\gamma\right)^{\frac{1}{p-1}}\left(1-e^{-\frac{\gamma}{p-1}t}\right)\right)^{\frac{p-1}{\gamma}}}. \]
Using the inequality $1-e^{-x}\leq x$ for $x>0$, we get
\[ \varrho_p(t,X_t(x)) \geq \frac{\varrho_0e^{-t}}{\left(1+\left(\frac{a}{\mu}\varrho_0^\gamma\right)^{\frac{1}{p-1}}\frac{\gamma}{p-1}t
\right)^{\frac{p-1}{\gamma}}}. \]
Note that
\[\begin{aligned} \left(1+\left(\frac{a}{\mu}\varrho_0^\gamma\right)^{\frac{1}{p-1}}\frac{\gamma}{p-1}t\right)^{\frac{p-1}{\gamma}} &\leq \max\left\{1,\left(\frac{a}{\mu}\|\varrho_0\|_{L^\infty}^\gamma\right)^{1/\gamma}\right\}\left(1+\frac{\gamma}{p-1}t\right)^{\frac{p-1}{\gamma}} \\
&\leq \max\left\{1,\left(\frac{a}{\mu}\|\varrho_0\|_{L^\infty}^\gamma\right)^{1/\gamma}\right\}e^t,
\end{aligned}\]
where in the second line we used the inequality $(1+x)^{\frac{1}{x}}\leq e$ for all $x>0$. Therefore
\[ \varrho_p(t,X_t(x)) \geq \frac{\varrho_0e^{-2t}}{\max\left\{1,\left(\frac{a}{\mu}\|\varrho_0\|_{L^\infty}^\gamma\right)^{1/\gamma}\right\}} \]
which provides the lower-bound.

\bigskip
\paragraph{Upper bound for the density.} 
In order to obtain an upper bound on the density we use in a crucial manner an
identity introduced in Basov-Shelukhin in \cite{basov-shelukin}. First we denote%
\[
\psi_{p}\left(  t,x\right)  =\int_{0}^{t}(\varrho_{p}u_{p}^{2}-\mu\left\vert
\partial_{x}u_{p}\right\vert ^{p-2}\partial_{x}u_{p}+a\varrho_{p}^{\gamma
})\;\dd s -\int_{\T}\Bigl(
\int_{y}^{x}(\varrho_{p,0}u_{p,0})\left(  z\right)
\mathrm{d}z\Bigr)dy.
\]
Observe that using the expression of $\psi_p$ and the momentum equation
\begin{align*}
\partial_{t}\psi_{p}\left(  t,x\right)   &  =\varrho_{p}u_{p}^{2}%
-\mu\left\vert \partial_{x}u_{p}\right\vert ^{p-2}\partial_{x}u_{p}%
+a\varrho_{p}^{\gamma},\\
\partial_{x}\psi_{p}\left(  t,x\right)   &  =\int_{0}^{t}(\varrho_{p}u_{p}%
^{2}-\mu\left\vert \partial_{x}u_{p}\right\vert ^{p-2}\partial_{x}%
u_{p}+a\varrho_{p}^{\gamma})_{x}-(\varrho_{p,0}u_{p,0})\left(  x\right)
=-\varrho_{p}u_{p}\left(  t,x\right)  ,
\end{align*}
from which, combining the two expression above, we obtain that%
\[
\partial_{t}\psi_{p}\left(  t,x\right)  +u_{p}\left(  t,x\right)  \partial
_{x}\psi_{p}\left(  t,x\right)  =-\mu\left\vert \partial_{x}u_{p}\right\vert
^{p-2}\partial_{x}u_{p}+a\varrho_{p}^{\gamma}.
\]
Observe also that%
\[
\psi_{p}\left(  t,x\right)  =\int_{\T}\psi_{p}\left(  t,y\right)
\mathrm{d}y+\int_{\T}
\Bigl(\int_{y}^{x}\partial_{x}\psi_{p}\left(  t,z\right)
\mathrm{d}z\Bigr)\dd y,
\]
such that using energy estimates
\[ |\psi_p(t,x)-\psi_p(0,x)|\leq (1+(2+\gamma)t)\overline{E_0} + \frac{\mu}{p}t \]
We now compute%
\begin{align}
&  \partial_{t}\left(  \varrho_{p}^{\mu}\exp\left(  -\psi_{p}\right)  \right)
+u_{p}\partial_{x}\left(  \varrho_{p}^{\mu}\exp\left(  -\psi_{p}\right)
\right)  \nonumber\\
&  =\varrho_{p}^{\mu}\exp\left(  -\psi_{p}\right)  \left(  -\mu\partial
_{x}u_{p}+\mu\left\vert \partial_{x}u_{p}\right\vert ^{p-2}\partial_{x}%
u_{p}-a\varrho_{p}^{\gamma}\right)  .\label{densitate_BasovShel}%
\end{align}
To estimate the right hand side of (\ref{densitate_BasovShel}), note that for for $\partial_x u_p>0$ we have
\[ -\mu\partial_xu_p + \mu|\partial_xu_p|^{p-2}\partial_xu_p-a\varrho_p^\gamma = \sigma_p-\mu\partial_xu_p \leq \mu, \]
where we used (\ref{bounded_flux}). Conversely, for $\partial_xu_p<0$ we see that
\[ -\mu\partial_xu_p+\mu|\partial_xu_p|^{p-2}\partial_xu_p-a\varrho_p^\gamma \leq \mu\partial_xu_p(|\partial_xu_p|^{p-2}-1) \leq \mu, \]
since the function $z\mapsto z(|z|^{p-2}-1)$ is bounded by $1$ for $z<0$. Therefore we have
\[ \partial_t(\varrho_p^\mu\exp(-\psi_p)) + u_p\partial_x(\varrho_p^\mu\exp(-\psi_p)) \leq \mu\varrho_p^\mu\exp(-\psi_p) \]
and using maximum principle and Gronwall's lemma we get
\[ \varrho_p(t,x) \leq \varrho_0(x)\exp\left(\frac{1}{\mu}(\psi_p(t,x)-\psi_p(0,x)) + t\right). \]
Therefore using the estimate on $\psi_p(t,x)-\psi_p(0,x)$ we end up with
\[ \varrho_p(t,x) \leq \varrho_0(x)\exp\left[\frac{\overline{E_0}}{\mu} + \left(\frac{2+\gamma}{\mu}\overline{E_0} + 1+\frac{1}{p}\right)t\right]. \]

\begin{rem}
At a first glance, it might seem imposible to obtain a bound for the density
using $\left(  \text{\ref{bounded_flux}}\right)  $ since this only gives a
control for the positive part of $\partial_{x}u_{p}$. Of course, the growth of
$\varrho_{p}$ is due to the negative part of $\partial_{x}u_{p}$ since%
\[
\partial_{t}\varrho_{p}+u_{p}\partial_{x}\varrho_{p}=-\varrho_{p}\partial
_{x}u_{p}.
\]
The identity $\left(  \text{\ref{densitate_BasovShel}}\right)  $ which is
inspired by \cite{basov-shelukin} is crucial since it shows that the growth of the modified
function $\varrho_{p}^{\mu}\psi_{p}$ is actually governed by $-\mu\partial
_{x}u_{p}+\mu\left\vert \partial_{x}u_{p}\right\vert ^{p-2}\partial_{x}%
u_{p}-a\varrho_{p}^{\gamma}.$ Another aspect that can be hiden in the details
is that the shear-thickening aspect i.e. $p\geq2$ is also crucial in order to
compensate the first term.
\end{rem}

\end{proof}

Next, we obtain estimates for the time derivatives using an idea introduced by Hoff \cite{Hoff1d,Hoff}. More precisely, denoting $\dot u_p = \partial_t u_p + u_p\partial_x u_p$, we obtain

\begin{prop}\label{prop_hoff} Let $(\varrho_p,u_p)$ be smooth solution of \eqref{p-system}. Then denoting 
$$ \int_0 ^t \int_{\T} \varrho_p |\dot u_p|^2
   + \frac{\mu}{2}  \int_{\T} \frac{|\partial_x u_p|^p}{p} (t)\leq C(t)$$
 where $C(t)$ does not depend on $p$.
 \end{prop} 

\begin{proof} The  momentum equation reads
$$\varrho_p \dot u_p 
- \mu \partial_x(
|\partial_x u_p|^{p-2} \partial_x u_p)
+ a \partial_x \varrho_p^\gamma
= 0$$
Multiplying by $\dot u_p$ and integrating in space, we get 
$$\int_{\T} \varrho_p |\dot u_p|^2 + \mu \frac{d}{dt} \int_{\T} \frac{|\partial_x u_p|^p}{p} 
+ \mu  \int_{\T} |\partial_x u_p|^{p-2} \partial_x u_p \partial_x(u_p\partial_x u_p) 
- a\int_{\T} \varrho_p^\gamma \partial_{tx}^2 u_p 
+ a\int_{\T} \partial_x (\varrho_p)^\gamma u_p \partial_x u_p = 0.
$$
First note that 
$$\int_{\T} |\partial_x u_p|^{p-2} \partial_x u_p \partial_x(u_p\partial_x u_p) 
 =  \int_{\T} |\partial_x u_p|^p \partial_x u_p 
 + \int_{\T} |\partial_x u_p|^{p-2} \partial_x u_p \partial_x^2 u_p u_p
 = (1-\frac{1}{p}) \int_{\T} |\partial_x u_p|^p \partial_x u_p.
 $$
 Then we can write
 $$\mu \int_{\T} |\partial_x u_p|^{p-2} \partial_x u_p\partial_x (u_p\partial_x u_p) 
 =  (1-\frac{1}{p})
  \int_{\T} |\partial_x u_p|^2 ( \sigma_p + a\varrho_p^\gamma).$$
 where we recall that $\sigma_p = \mu |\partial_x u_p|^{p-2} \partial_x u_p - a\varrho_p^\gamma$.
  
\noindent Secondly, observe that
  $$- \int_{\T} \varrho_p^\gamma \partial_{tx}^2 u_p
  = - \frac{d}{dt} \int_{\T} \varrho_p^\gamma \partial_x u_p
      + \int_{\T} \partial_t (\varrho_p^\gamma) \partial_x u_p$$
and
  $$ \partial_t (\varrho_p^\gamma) + u_p \partial_x (\varrho_p^\gamma)
     + \gamma \varrho_p^\gamma \partial_x u_p = 0$$
which implies
  $$\int_{\T} \partial_t \varrho_p^\gamma \partial_x u_p
  + \int_{\T} u_p \partial_x (\varrho_p^\gamma) \partial_x u_p 
   = - \int_{\T}  \gamma \varrho_p^\gamma |\partial_x u_p|^2$$
to conclude that
$$- \int_{\T} \varrho_p^\gamma \partial_{tx}^2 u_p
+ \int_{\T} \partial_x (\varrho_p)^\gamma u_p \partial_x u_p 
= - \frac{d}{dt} \int_{\T} \varrho_p^\gamma \partial_x u_p 
   - \int_{\T} \gamma \varrho_p^\gamma |\partial_x u_p|^2.
$$
Plugging everything together we can write
\begin{multline*}\int_{\T} \varrho_p |\dot u_p|^2 + \mu \frac{d}{dt} \int_{\T} \frac{|\partial_x u_p|^p}{p} 
+  (1-\frac{1}{p})
  a\int_{\T} |\partial_x u_p|^2  \varrho_p^\gamma \\
=
 \frac{d}{dt} \int_{\T} a\varrho_p^\gamma \partial_x u_p 
   +  a\int_{\T} \gamma \varrho_p^\gamma |\partial_x u_p|^2
   - (1-\frac{1}{p}) \int_{\T} \sigma_p |\partial_x u_p|^2 \end{multline*}
Integrating in time from $0$ to $t \in (0,T)$, we get
$$\int_0^t \int_{\T} \varrho_p |\dot u_p|^2 
  + \mu  \int_{\T} \frac{|\partial_x u_p|^p}{p}(t) 
+  a(1-\frac{1}{p})
 \int_0^t  \int_{\T} |\partial_x u_p|^2  \varrho_p^\gamma
=
 \int_{\T} (a\varrho_p^\gamma \partial_x u_p) (t)
-  \int_{\T} (a\varrho_p^\gamma \partial_x u_p) (0)
$$
$$
   + \int_0^t \int_{\T} \gamma a \varrho_p^\gamma |\partial_x u_p|^2
   - (1-\frac{1}{p}) \int_0^t \int_{\T} \sigma_p |\partial_x u_p|^2 
   +\mu  \int_{\T} \frac{|\partial_x (u_p)_0|^p}{p}  
 = \sum_{i=1}^5 I_i 
$$
Concerning $I_4$,  recall that 
$$\partial_x \sigma_p = \varrho_p \dot u_p$$
Since $\int_{\T} \partial_x u_p (t,y) dy = 0$ due to the periodic boundary condition,
we get  by the mean value theorem that for all $t$, there exists $x(t)$ such that
$\partial_x u_p(t,x(t))=0$.
Therefore integrating the momentum equation from $x(t)$ to $x$, we get 
$$\sigma_p (t,x) =     \int_{x(t)}^x ( \varrho_p \dot u_p) (t,z) dz    -a\varrho_p^\gamma(t,x(t))$$
We can thus treat $I_3$ and $I_4$ together using that
$$ xy \le \frac{2}{p} x^{p/2} + \frac{p-2}{p} y^{p/(p-2)}$$
choosing $x = |\partial_x u_p|^2$ and $y=\varrho_p^\gamma$ for $I_3$ (recalling
that $y$ is uniformly bounded) and $y= \sigma_p$ for $I_4$ using the bound 
given above. 
Concerning $I_1$, we note that for all $t\in (0,T]$, 
$$\int_{\T} (\varrho_p^\gamma \partial_x u_p)(t) 
\le \frac{\mu}{2 p}  \int_{\T} |\partial_x u_p|^p(t)  + \frac{p}{p-1}  (\frac{2}{\mu})^{(p-1)/p^2} 
 \int_{\T} |\varrho_p^\gamma|^{(p-1)/p}(t)
 \le  \frac{\mu}{2 p}  \int_{\T} |\partial_x u_p|^p(t)
    + C.
$$
In conclusion denoting 
$$Y = \int_0 ^t \int_{\T} \varrho_p |\dot u^p|^2
   + \frac{\mu}{2}  \int_{\T} \frac{|\partial_x u_p|^p}{p} (t),$$
we get
 $$ Y(t) \le C_1 + C_2 \int_0^t Y(\tau) d\tau$$
 where $C_1$ and $C_2$ are constants that do not depend on $p$. The assertion follows then from Gronwall's lemma.
\end{proof}

\bigskip

\bigskip

An important consequence of Proposition \ref{prop_hoff} is the uniform-in-$p$ bound on $|\partial_xu_p|^{p-2}\partial_xu_p$:

\begin{prop}  Let $(\varrho_0,u_0)$ be smooth with $0 \le c \le \varrho_0 \le C <+\infty$. Let $(\varrho_p,u_p)$ be smooth solution of \eqref{p-system}. Then 
$$|\partial_x u_p|^{p-2} \partial_x u_p \in L^2_t L^\infty_x.$$
 \end{prop} 

\begin{proof} As explained in the proof of the previous proposition, the momentum equation gives%
\[\mu\left\vert \partial_{x}u_{p}\right\vert ^{p-2}\partial_{x}u_{p}\left(
t,x\right)  =a\varrho_{p}^\gamma \left(  t,x\right)  
+ \int_{x(t)}^x (\varrho \dot u_p)(t,z) dz -a\varrho_p^\gamma(t,x(t))
\]
and therefore using the bound on $\varrho_p$
and $\dot u_p$, we get
the conclusion.
\end{proof}

 Finally, all the convergence results announced in \eqref{convp}, except the last one regarding the density are consequences of the apriori estimates obtained so for. 

\subsection{Asymptotic convergence with respect to $p$.}  
We first use the estimates to pass to the limit in the variational inequality and then show that we can pass to the limit in the equation as well. 

\noindent 

\subsubsection{Strong convergence of the density $\varrho_p$.}
We strongly use the bounds we derived on $\varrho_p$, $\dot u_p$ and $\tau_p$ to justify the calculations we make below.
Let us multiply the momentum equation
satisfied by $(\varrho_p,u_p)$ by $u_p$,
we get
$$\frac{a}{\gamma-1} \int_{\T}  
 \varrho_p^\gamma (t,x) dx
 + \int_0^t \int_{\T} [(\varrho_p \dot u_p) u_p] (\tau,x) 
dx d\tau
 $$
$$\hskip6cm
 + \int_0^t\int_{\T}|\partial_x u_p|^p(\tau,x) dx d\tau 
 = \frac{a}{\gamma-1} \int_{\T}  
 \varrho_0^\gamma (x) dx
$$ 
and testing it also with $u$, we get
$$\frac{a}{\gamma-1} \int_{\T}  
 \varrho^\gamma (t,x) dx
-  a \int_0^t \int_{\T} 
  [(\varrho_p^\gamma - \varrho^\gamma) \partial_x u](\tau,x) dxd\tau \\
$$
$$
+ \int_0^t \int_{\T} 
 [(\varrho_p \dot u_p) u](\tau,x) dx d\tau
$$
$$
+ \int_0^t \int_{\T} [|\partial_x u_p|^{p-2} \partial_x u_p \, \partial_x u](\tau,x) dxd\tau
= \frac{a}{\gamma-1} \int_{\T}  
 \varrho_0^\gamma (x) dx.
$$
and therefore substracting the two equalities, we get
$$\frac{a}{\gamma-1} \int_{\T} 
(\varrho_p^\gamma - \varrho^\gamma)(t,x) dx
+ a \int_0^t \int_{\T} [(\varrho_p^\gamma
- \varrho^\gamma) \partial_x u](\tau,x) dxd\tau \\
$$
$$
+ \int_0^t \int_{\T}
   [(\varrho_p \dot u_p )(u_p-u)]
     (\tau,x)dx d\tau
     $$
     $$
+ \int_0^t \int_{\T} |\partial_x u_p|^{p-2} \partial_x u_p (\partial_x u_p - \partial_x u) = 0.
$$
We now use the fact that by monotonicity
$$ \bigl(|\partial_x u_p|^{p-2} \partial_x u_p -
|\partial_x u|^{p-2} \partial_x u
\bigr) (\partial_x u_p - \partial_x u)
\ge 0$$
to deduce that
$$\frac{a}{\gamma-1} \int_{\T} 
(\varrho_p^\gamma - \varrho^\gamma)(t,x) dx
+ a \int_0^t \int_{\T} [(\varrho_p^\gamma
- \varrho^\gamma) \, \partial_x u](\tau,x) dxd\tau \\
$$
$$
+ \int_0^t \int_{\T}
   [(\varrho_p \dot u_p)(u_p-u)]
     (\tau,x)dx d\tau
     $$
     $$
+ \int_0^t \int_{\T} |\partial_x u|^{p-2} \partial_x u (\partial_x u_p - \partial_x u) \le 0.
$$
Note that we know that $|\partial_x u|\le 1$ and that we also know that
$u_p \to u$ in $L^2$ by Aubin-Lions using the bound  on $\partial_t u_p$ in $L^2$ and the bound on $u_p$ in $L^p_t W^{1,p}_x$. We also have weak convergence in $L^2$ on $\varrho_p \dot u_p$ and the weak convergence of $\partial_x u_p$ to $\partial_x u$ in 
$L^2$. Writing
$$\int_{\T} 
(\varrho_p^\gamma - \varrho^\gamma)(t,x) dx
= \int_{\T} 
(\varrho_p^\gamma - \varrho^\gamma
- \gamma \varrho^{\gamma-1}(\varrho_p-\varrho))(t,x) dx
+ \gamma \int_{\T} (\varrho^{\gamma-1} (\varrho_p-\varrho))(t,x) dx.
$$
and
$$\int_0^t \int_{\T} 
[(\varrho_p^\gamma - \varrho^\gamma )\partial_x u](\tau,x) dx d\tau
= \int_0^t\int_{\T} 
[(\varrho_p^\gamma - \varrho^\gamma - \gamma \varrho^{\gamma-1} (\varrho_p-\varrho))\partial_x u
](\tau,x) dx d\tau$$
$$
+ \gamma \int_0^t\int_{\T} (\varrho^{\gamma-1} (\varrho_p-\varrho)\partial_x u)(\tau,x) dx d\tau.
$$
we can conclude obtaining an
inequality 
$$X_p \le \int_0^t X_p + \varepsilon_p(t)
\hbox{ where } \displaystyle X_p= \int_{\T} [\varrho_p^\gamma - \varrho^\gamma - \gamma \varrho^{\gamma-1} (\varrho_p-\varrho)](t,x) dx$$
with some rest $\varepsilon_p(t)$ which converges to zero a.e. as $p\to \infty$. Using that $\varrho_p^\gamma - \varrho^\gamma - \gamma \varrho^{\gamma-1} (\varrho_p-\varrho) \ge 0$, This allows to conclude to the strong convergence of $\varrho_p$ in $L^\gamma((0,T)\times
(0,1))$ and then in ${\mathcal C}(0,T; L^p(0,1))$ for all $p<+\infty$ using the uniform bound with respect to $p$ and the mass equation on $\partial_t\varrho_p$.

\subsubsection{Asymptotic limit on a weak formulation inequality}
Let us now use the fact that $\varrho_p$ strongly converges to $\varrho$ in ${\mathcal C}(0,T; L^p(0,1))$ for all $p<+\infty$.
From the previous estimates we note that we have the convergence written in Theorem \ref{MainTh1}.
We can then pass to the limit in the variational inequality which is satisfied by a global
weak solution of the $p$-laplacian compressible Navier--Stokes equations, namely 
$$\int_0^T\int_{\T} \varrho_p \partial_t u_p (v-u_p)
   + \varrho_p u_p \partial_x u_p (v-u_p) 
   - \int_0^T\int_{\T} |\partial_x v|^{p-2}\partial_x v \, (\partial_x v - \partial_x u_p) 
   - \int_{\T} \varrho_p^\gamma (\partial_x v- \partial_x u_p) \ge 0
$$
with $v\in L^\infty(0,T;W^{1,r}(0,1))$ for all $r<+\infty$ satisfying $|\partial_x v| \le 1$.
This follows the same lines as in the paper by J.F. Rodrigues {\it et al.} \cite{Rodrigues,RodSan}
using
the convergence deduced from the bounds \eqref{convp}.  Note that the third quantity tends to
zero taking test functions satisfying  $|\partial_x v|< 1$. We recall the weak convergence of
$\partial_t u_p$ in $L^2$, the strong convergence of $\varrho_p$ in any $L^p$ because of the strong convergence in $L^2$ we make here and the uniform bound on $\varrho_p$ in $L^\infty$. We also
have strong convergence on $u_p$ because of the uniform bound we have on $\partial_x u$
in any $L^q$ for $q\le p$ and $\partial_t u_p$ in $L^2$.
We get  the  constraint on $|\partial_x u|$ by
showing the following Proposition:
\begin{prop}\label{prop_dx_u} For all $r\in\left(  1,\infty\right)$ we have
\[
\partial_{x}u_{p}\rightharpoonup\partial_{x}u\text{ in }L_{t,x}^{r}\text{.}%
\]
Moreover, it turns out that $\partial_{x}u\in L_{t,x}^{\infty}$ with
$\left\vert \partial_{x}u\right\vert \leq1$. 
\end{prop} 

\begin{proof} 
From the last two conditions and the bound on the derivative of $\partial
_{x}u_{p}$ we can deduce an uniform bound on $u_{p}$. 
For reader's convenience, we recall the proof
\begin{align*}
&  \left(  1+\eta\right)  \left\vert \left\{  (t,x):\left\vert \partial
_{x}u\right\vert \geq1+\eta\right\}  \right\vert \\
&  \leq\int_{0}^{t}\int_{\T}\left\vert \partial_{x}u\right\vert 1_{\left\{
(t,x):\left\vert \partial_{x}u\right\vert \geq1+\eta\right\}  }\leq\lim\inf%
{\displaystyle\iint\limits_{\left\{  (t,x):\left\vert \partial_{x}u\right\vert
\geq1+\eta\right\}  }}
\left\vert \partial_{x}u_{p}\right\vert \\
&  \leq\left\vert \left\{  (t,x):\left\vert \partial_{x}u\right\vert
\geq1+\eta\right\}  \right\vert ^{1-\frac{1}{p}}\lim\inf\left(  \int_{0}%
^{t}\int_{\T}\left\vert \partial_{x}u_{p}\right\vert ^{p}\right)
^{\frac{1}{p}}\\
&  \leq\left\vert \left\{  (t,x):\left\vert \partial_{x}u\right\vert
\geq1+\eta\right\}  \right\vert ^{1-\frac{1}{p}}\left(  \int_{\T}%
(\varrho_{0}\log\varrho_{0}-\varrho_{0})\mathrm{d}x\right)  ^{\frac{1}{p}}%
\end{align*}
such that for $p\rightarrow\infty$ we get that for all $\eta>0$ :
\[
\left\vert \left\{  (t,x):\left\vert \partial_{x}u\right\vert \geq
1+\eta\right\}  \right\vert =0.
\]
\end{proof}

\subsubsection{Asymptotic limit in the strong form equality} 
We will use that $\varrho_p$ strongly converges to $\varrho$ in ${\mathcal C}(0,T; L^p(\T))$ for all $p<+\infty$.  This part is dedicated to a proof which is largely inspired by the result 
in \cite[Theorem 2.3.]{DPEvPr}.
 Denoting $\tau_p = |\partial_x u_p|^{p-2} \partial_x u_p$, we have
$$\partial_x \tau_p
= a \partial_x         \varrho_p^\gamma
  + \varrho_p \dot u_p = f_p.$$
We know that $\tau_p$ weakly converges to $\tau$ in $L^1_{t,x}$
using the extra integrability on $\tau_p$. Note that we have
$$ \int_0^T\int_{\T}  |\tau| \le 
\liminf_{p\to +\infty}  \int_0^T\int_{\T}  |\tau_p|
\le \liminf_{p\to \infty}  \int_0^T\int_{\T}  |\partial_x u_p|^{p-1}
\le \liminf_{p\to \infty} \bigl( \int_0^T\int_{\T}  |\partial_x u_p|^{p}
\bigr)^{(p-1)/p}
$$
Note that
$$  \int_0^T\int_{\T}  |\partial_x u_p|^p 
= < f_p, u_p>_{L^2(0,T;H^{-1}(\T))\times L^2(0,T;H^1(\T))}
$$
and that we can prove using the convergence in theorem \ref{MainTh1} that
$$<f_p, u_p>_{L^2(0,T;H^{-1}(\T))\times L^2(0,T;H^1(\T))}
\to <f,u>_{L^2(0,T;H^{-1}(\T))\times L^2(0,T;H^1(\T))}$$
where
$$ f = a\partial_x \varrho^\gamma + \varrho \dot u
$$ 
and with 
$$<f,u>_{L^2(0,T;H^{-1}(\T))\times L^2(0,T;H^1(\T))}= \int_0^T\int_{\T}  \tau \partial_x u.$$
Thus we have proved that
$$\int_0^T\int_{\T} |\tau| \le 
\int_0^T\int_{\T} \tau \partial_x u.
$$
We remind that we have already proved that $$
|\partial_x u|\le 1 
\hskip1cm {\it a.e.}$$ 
and therefore we get
$$|\tau| = \tau \partial_x u .$$
Assuming $|\tau|\not = 0$, then multiplying 
the relation by $\tau/|\tau|$ we get
$$\tau = |\tau| \partial_x u$$ 
relation that extend for $|\tau|=0$.
Denoting $\pi = |\tau|$, we then conclude that 
$$\tau = \pi \partial_x u
\qquad
\hbox{ with }
\qquad (|\partial_x u|-1)\, \pi=0,
$$
which in turn allows us to obtain (\ref{limit-system}).

\section{The multi-dimensional semi stationary case}
The current section is devoted to prove Theorems \ref{th_multid} and \ref{th_multid_ex}.

\begin{proof}[Proof of Theorem \ref{th_multid}.]
    From the uniform estimate (\ref{energy}), up to a subsequence we have the weak convergence
    \begin{equation}\label{convergence}\begin{aligned} \varrho_p\rightharpoonup^*\varrho \quad &\text{in} \quad L^\infty(0,T;L^\gamma), \\
u_p\rightharpoonup u \quad &\text{in} \quad L^d(0,T;W^{1,d}), \\
\varrho_p^\gamma \rightharpoonup^* \overline{\varrho^\gamma} \quad &\text{in} \quad L^\infty(0,T;\mathcal{M}).
\end{aligned}\end{equation}

Moreover, we have the following estimate: 
\begin{prop}\label{Du_bdd} The limit $u$ satisfies $|\D u|\leq 1$ a. e. 
\end{prop}

\begin{proof}
    This is the multidimensional equivalent of the estimate from Proposition \ref{prop_dx_u}. Assume that there exist $\eta>0$ such that
    \[ |\{(t,x):|\D u|>1+\eta\}| > 0. \]
    From the Chebyshev inequality, we have
    \[\begin{aligned} |\{|\D u|>1+\eta\} \leq & \frac{1}{1+\eta}\int_0^T\int_{\T^d} |\D u|\mathbbm{1}_{|\D u|>1+\eta}\;\dd x \leq \frac{1}{1+\eta}\liminf_{p\to\infty}\int_0^T\int_{\T^d} |\D u_p|\mathbbm{1}_{|\D u|>1+\eta} \;\dd x \\
    \leq & \frac{1}{1+\eta}\liminf_{p\to\infty}\left(\|\D u_p\|_{L^p}|\{|\D u|>1+\eta\}|^{1-1/p}\right). \end{aligned}\]
    Since $\|\D u_p\|_p\leq \left(\frac{1}{\gamma-1}\int_{\T^d}\varrho_0^\gamma\;\dd x\right)^{1/p}\leq C^{1/p}$, dividing the both sides by $|\{|\D u|>1+\eta\}|$, we get
    \[ 1\leq \frac{1}{1+\eta}\liminf_{p\to\infty}\left(C^{1/p}|\{|\D u|>1+\eta\}|^{-1/p}\right) = \frac{1}{1+\eta}. \]
    Since $\eta>0$, we got a contradiction and thus $|\{|\D u|>1\}|=0$.
\end{proof}

With this information at hand, we will show that $\varrho_p\to\varrho$ strongly.
First, from the continuity equation, up to a subsequence we have $\varrho_p\to\varrho$ in $C_{weak}(0,T;L^\gamma)$, and in consequence
\[ \varrho_p\to\varrho \quad \text{in} \quad C(0,T;W^{-1,\frac{d}{d-1}}). \]
In particular, $\overline{\varrho u}=\varrho u$ and $(\varrho,u)$ satisfy the continuity equation
\[ \partial_t\varrho + \ddiv(\varrho u) = 0 \]
in the renormalized sense (since $\nabla u$ is integrable up to any power). In conclusion, we also have
\[ \partial_t\varrho^\gamma + \ddiv(\varrho^\gamma u) + (\gamma-1)\varrho^\gamma\ddiv u = 0 \]
in the sense of distributions. Since $|\ddiv u|\leq 1$ a. e., we can also improve the regularity of $\varrho$:
\begin{prop}\label{rho_bdd} The limit density $\varrho$ satisfies the estimate
\[ \|\varrho\|_{L^\infty([0,T]\times\T^d)} \leq \|\varrho_0\|_{L^\infty(\T^d)}e^t. \]
\end{prop}
\begin{proof} 
    Let $T_k\in C^\infty([0,\infty))$ be a truncation operator, satisfying $T_k(z)=z$ for $z\leq k$, $T_k(z)=k+1$ for $z>2k$, $T_k'(z)>0$ and $T_k(z)\nearrow z$ when $k\to\infty$. Define
    \[ P_k(\varrho) = \varrho\int_0^\varrho\frac{(T_k(z))^q}{z^2}\;\dd z \quad \text{for} \quad q>1. \]
    It is easy to see that 
    \[ P_k(\varrho)\geq \left(\frac{k}{k+1}\right)^{q-1}\frac{1}{q-1}(T_k(\varrho))^q. \]
    For $\varrho\leq k$ we have $P_k(\varrho)=\frac{1}{q-1}\varrho^q = \frac{1}{q-1}(T_k(\varrho))^q$, and for $\varrho>k$
    \[ P_k(\varrho) \geq \varrho\int_0^k \frac{(T_k(z))^q}{z^2}\;\dd z = \frac{\varrho}{q-1}k^{q-1} = \frac{1}{q-1}\varrho(k+1)^{q-1}\left(\frac{k}{k+1}\right)^{q-1}\geq \left(\frac{k}{k+1}\right)^{q-1}\frac{1}{q-1}(T_k(\varrho))^q. \]
    Since $\varrho$ satisfies the continuity equation in the renormalized sense, it holds
    \[ \partial_t P_k(\varrho) + \ddiv(P_k(\varrho)u)+(T_k(\varrho))^q\ddiv u = 0. \]
    Integrating it over $\T^d$, we get
    \[ \frac{\dd}{\dd t}\int_{\T^d}P_k(\varrho)\;\dd x = -\int_{\T^d}(T_k(\varrho))^q\ddiv u\;\dd x \leq \int_{\T^d}(T_k(\varrho))^q\;\dd x \]
    and in consequence
    \[ \int_{\T^d}P_k(\varrho)\;\dd x \leq \int_{\T^d} P_k(\varrho_0)\;\dd x + \int_0^t\int_{\T^d}(T_k(\varrho))^q\;\dd x\dd s. \]
    Therefore, using the relation between $T_k$ and $P_k$, we get
    \[ \int_{\T^d}(T_k(\varrho))^q\;\dd x \leq (q-1)\left(\frac{k+1}{k}\right)^{q-1}\int_{\T^d}P_k(\varrho_0)\;\dd x + (q-1)\left(\frac{k+1}{k}\right)^{q-1}\int_0^t\int_{\T^d}(T_k(\varrho))^q\;\dd x \]
    and from Gronwall's lemma
    \begin{equation}\label{gronwall_k} \int_{\T^d} (T_k(\varrho))^q\;\dd x \leq (q-1)\left(\frac{k+1}{k}\right)^{q-1}\int_{\T^d}P_k(\varrho_0)\;\dd x \cdot\exp\left((q-1)\left(\frac{k+1}{k}\right)^{q-1}t\right). \end{equation}
    From the Monotone Convergence Theorem, we know that
    \[ \lim_{k\to\infty}\int_{\T^d}(T_k(\varrho))^q\;\dd x = \int_{\T^d}\varrho^q\;\dd x \quad \text{for a. a. } t\in[0,T]. \]
    Moreover, by the Dominated Convergence Theorem  it holds
    \[ \lim_{k\to\infty}\int_{\T^d} P_k(\varrho_0)\;\dd x = \frac{1}{q-1}\int_{\T^d}\varrho_0^q\;\dd x. \]
    Therefore passing to the limit with $k$ in (\ref{gronwall_k}) we obtain
    \[ \|\varrho(t,\cdot)\|_{L^q(\T^d)}\leq \|\varrho_0\|_{L^q(\T^d)}e^{(1-1/q)t} \quad \text{for a. a. } t\in [0,T], \]
    and since $\varrho_0\in L^\infty(\T^d)$, passing to the limit with $q\to\infty$ we get the bound
    \[ \|\varrho(t,\cdot)\|_{L^\infty(\T^d)}\leq \|\varrho_0\|_{L^\infty(\T^d)}e^t. \]
\end{proof}
Having the $L^\infty$ bounds on $\varrho$ and $\ddiv u$, we can apply the final argument. We test (\ref{eq_p}) by $\psi(t)u_p$, where $\psi\in C_0^\infty(0,T)$, $\psi\geq 0$. Then, we get
\begin{equation}\label{energy1}
   \int_0^T\psi(t)\int |\D u_p|^p\;\dd x\dd t -\frac{1}{\gamma-1}\int_0^T\psi'(t)\int \varrho_p^\gamma\;\dd x\dd t = 0.
\end{equation}
On the other hand, testing (\ref{eq_p}) by the (localized in time) limit $\psi(t)u$, we get
\begin{equation}\label{energy2} \int_0^T\psi(t)\int |\D u_p|^{p-2}\D u_p:\D u\;\dd x\dd t - \int_0^T\psi(t)\int \varrho_p^\gamma\ddiv u\;\dd x\dd t = 0. \end{equation}
From the monotonicity of the function $\mathbb{S}_p(\D u)=|\D u|^{p-2}\D u$, we get the inequality:
\[\begin{aligned} |\D u_p|^p-|\D u_p|^{p-2}\D u_p:\D u = & \mathbb{S}_p(\D u_p):(\D u_p-\D u) \\
= & (\mathbb{S}_p(\D u_p)-\mathbb{S}(\D u)):(\D u_p-\D u) + \mathbb{S}_p(\D u):(\D u_p-\D u) \\
\geq & \mathbb{S}_p(\D u):(\D u_p-\D u). \end{aligned}\]
Therefore substracting (\ref{energy2}) from (\ref{energy1}) and using the equation on $\varrho^\gamma$, we get
\begin{multline}\label{main_ineq} \int_0^T\psi(t)\int \mathbb{S}_p(\D u):(\D u_p-\D u)\;\dd x\dd t -\frac{1}{\gamma-1}\int_0^T\psi'(t)\int \varrho_p^\gamma-\varrho^\gamma\;\dd x\dd t \\
\leq -\int_0^T\psi(t)\int (\varrho_p^\gamma-\varrho^\gamma)\ddiv u\;\dd x\dd t. \end{multline}
Using the convexity of a function $\varrho\mapsto\varrho^\gamma$, we know that
\[ \varrho_p^\gamma - \varrho^\gamma - \gamma\varrho^{\gamma-1}(\varrho_p-\varrho) \geq 0. \]
Thus, we rewrite (\ref{main_ineq}) as
\[ -\frac{1}{\gamma-1}\int_0^T\psi'(t) X_p(t)\;\dd t \leq \int_0^T\psi(t) \left(\|\ddiv u\|_{L^\infty}X_p(t) + \varepsilon_p(t)\right)\;\dd t -\frac{1}{\gamma-1}\int_0^T\psi'(t)g_p(t)\;\dd t, \]
where 
\[ X_p(t) = \int \varrho_p^\gamma-\varrho^\gamma - \gamma\varrho^{\gamma-1}(\varrho_p-\varrho)\;\dd x \geq 0, \]
\[ \varepsilon_p(t) = - \gamma\int \varrho^{\gamma-1}(\varrho_p-\varrho)\ddiv u\;\dd x - \int \mathbb{S}_p(\D u):(\D u_p-\D u)\;\dd x, \]
and
\[ g_p(t) = \gamma\int \varrho^{\gamma-1}(\varrho_p-\varrho)\;\dd x. \]

For a fixed $t\in (0,T)$, choosing $\psi(s)=\eta_\delta(t-s)$, where $\eta_\delta$ is a standard mollifier, we get
\[ (X_p)_\delta'(t) \leq a (X_p)_\delta(t) + (\gamma-1)(\varepsilon_p)_\delta(t) +(g_p)_\delta'(t),  \]
where $a=(\gamma-1)\|\ddiv u\|_{L^\infty([0,T]\times\Omega)}$ and $(f)_\delta=f\ast\eta_\delta$. Therefore, from Gronwall's lemma, for $0< s<t<T$ we get
\[ (X_p)_\delta(t) \leq (X_p)_\delta(s)e^{a(t-s)} + (g_p)_\delta(t)-(g_p)_\delta(s)e^{a(t-s)} + \int_s^t \big((\gamma-1)(\varepsilon_p)_\delta(\tau) + a (g_p)_\delta(\tau)\big)e^{a(t-\tau)}\;\dd\tau.  \]
Since $X_p,\varepsilon_p\in L^1(0,T)$ and $g_p\in C(0,T)$, after passing to the limit with $\delta$ we get for almost all $0<s<t<T$
\[ X_p(t) \leq X_p(s)e^{a(t-s)} + g_p(t)-g_p(s)e^{a(t-s)} + \int_s^t \big((\gamma-1)\varepsilon_p(\tau)+ag_p(\tau)\big)e^{a(t-\tau)}\;\dd\tau. \]

Taking $\frac{1}{h}\int_0^h\;\dd s$ on both sides, we further have
\[\begin{aligned} X_p(t) \leq & \frac{1}{h}\int_0^h X_p(s)e^{a(t-s)}\;\dd s + g_p(t)-\frac{1}{h}\int_0^hg_p(s)e^{a(t-s)}\;\dd s \\
&+ \frac{1}{h}\int_0^h\int_s^t \big((\gamma-1)\varepsilon_p(\tau)+ag_p(\tau)\big)e^{a(t-\tau)}\;\dd\tau\dd s. \end{aligned}\]
Using Lemma \ref{transport_lemma}, we have
\[ \lim_{h\to 0}\frac{1}{h}\int_0^h X_p(s)\;\dd s = \lim_{h\to 0}\frac{1}{h}\int_0^h\int (\varrho_p^\gamma-\varrho_0^\gamma)\;\dd x\dd s + \lim_{h\to 0}\int_0^h\int (\varrho_0^\gamma-\varrho^\gamma)\;\dd x\dd s - \lim_{h\to 0}\int_0^h g_p(s)\;\dd s =0. \]
Therefore passing to the limit with $h$, we end up with
\[ X_p(t) \leq g_p(t) + \int_0^t \big((\gamma-1)\varepsilon_p(\tau)+ag_p(\tau)\big)e^{a(t-\tau)}\;\dd\tau. \]
Since $\varrho_p\to\varrho$ in $C_{\textrm{weak}}(0,T;L^\gamma)$ and $\varrho\in L^\infty([0,T]\times\Omega)$, we have for all $t\in [0,T]$
\[ g_p(t)=\int \gamma\varrho^{\gamma-1}(\varrho_p-\varrho)(t,\cdot)\;\dd x \to 0 \]
as $p\to\infty$. Moreover, 
\[ \varepsilon_p(t)=\int_0^t\int \mathbb{S}_p(\D u):(\D u_p-\D u) \;\dd x\dd s \to 0, \]
as $\mathbb{S}_p(\D u)=|\D u|^p\to \mathbbm{1}_{|\D u|=1}$ strongly as $p\to\infty$ and $\D u_p\to\D u$ weakly.

In conclusion, we get that for a. a. $t\in[0,T]$
\[ \lim_{p\to\infty}X_p(t) = 0 \]
and thus, since $X_p=\int (\varrho_p^\gamma-\varrho^\gamma)\;\dd x - g_p(t)$, we have
\[ \lim_{p\to\infty}\int \varrho_p^\gamma\;\dd x = \int \varrho^\gamma\;\dd x, \]
which, together with the weak convergence $\varrho_p\rightharpoonup^*\varrho$ in $L^\infty(0,T;L^\gamma)$, implies the strong convergence of $\varrho_p$ in $L^\gamma(0,T;L^\gamma)$.
\end{proof}

\subsection{Construction of solutions (Proof of Theorem \ref{th_multid_ex})}

In the end of this section, we will prove Theorem \ref{th_multid_ex}, by using the global existence result from \cite{FLM} and the compactness from Theorem \ref{th_multid}.

\begin{proof}[Proof of Theorem \ref{th_multid_ex}]
Fix $\varepsilon, b>0$ and let
    \[ \lambda(z) = \frac{\varepsilon b}{(1-(bz)^a)^{1/a}}. \]
In order to construct a solution, we use the result from \cite{FLM}, where the authors analysed the system
\begin{equation}\label{NS_lambda}\begin{aligned}
\partial_t\varrho + \ddiv(\varrho u) &= 0, \\
\partial_t(\varrho u) + \ddiv(\varrho u\otimes u) - \ddiv\left(2\mu(|\D^du|)\D^du\right) - \nabla(\lambda(|\ddiv u|)\ddiv u) + \nabla\varrho^\gamma &= 0,
\end{aligned}\end{equation}
where $\D^du = \D u-\frac{1}{d}\ddiv u\mathbb{I}$ is the deviatoric part of $\D u$ and
\[ \mu(|\D^du|)=\mu_0(1+|\D^du|^2)^{\frac{p-2}{2}}. \]
They prove that for a sufficiently large $p$ there exists a global variational solution to (\ref{NS_lambda}), i. e. $(\varrho,u)$ satisfies the continuity equation and for each $\varphi\in C_0^\infty([0,T]\times\T^d)$
    \begin{equation}\label{NS_lambda_variational}
    \begin{aligned} \int_0^t\int_{\T^d}\Lambda(|\ddiv \varphi|)\;\dd x\dd s \geq & \int_0^t\int_{\T^d}\Lambda(|\ddiv u|)\;\dd x\dd s + \left[\frac{1}{2}\int_{\T^d}\varrho|u|^2\;\dd x\right]\Big|_0^t - \left[\int_{\T^d}\varrho u\cdot\varphi\;\dd x\right]\Big|_0^t \\
    &+ \int_0^t\int_{\T^d}\varrho u\cdot\partial_t\varphi + \varrho u\otimes u:\nabla\varphi;\dd x\dd s \\
    &+ \int_0^t\int_{\T^d}\mu(|\D^d u|)\D^d u:(\D^d u-\D^d\varphi)\;\dd x\dd s \\
    &- \int_0^t\int_{\T^d} \varrho^\gamma(\ddiv u-\ddiv\varphi)\;\dd x\dd s, \end{aligned}\end{equation}
    where $\Lambda$ is given by
    \[ \left\{\begin{aligned}
        \Lambda'(z)=z\lambda(z) \quad \text{for}& \quad |z|<\frac{1}{b}, \\
        \Lambda(z) = +\infty \quad \text{for}& \quad |z|\geq\frac{1}{b} 
    \end{aligned}\right. \]
    
Note that although the setting in \cite{FLM} is different than in our case, from the proof it is clear that the same result holds without the presence of the convective term and if we replaced $\D^d u$ by $\D u$ and put $\mu(|\D u|)=|\D u|^{p-2}$. In particular, there exists a variational solution to
\begin{equation}\label{lambda}
        \begin{aligned}
            \partial_t\varrho + \ddiv(\varrho u) &= 0, \\
            -\ddiv(|\D u|^{p-2}\D u) - \nabla(\lambda(|\ddiv u|)\ddiv u) + \nabla\varrho^\gamma &= 0,
        \end{aligned}
    \end{equation}
and the variational inequality in this case has the form
\begin{equation}\label{lambda_variational}
    \begin{aligned} \int_0^t\int_{\T^d}\Lambda(|\ddiv \varphi|)\;\dd x\dd s \geq & \int_0^t\int_{\T^d}\Lambda(|\ddiv u|)\;\dd x\dd s + \int_0^t\int_{\T^d}|\D u|^{p-2}\D u:(\D u-\D\varphi)\;\dd x\dd s \\
    &- \int_0^t\int_{\T^d} \varrho^\gamma(\ddiv u-\ddiv\varphi)\;\dd x\dd s. \end{aligned}\end{equation}
Having the solution $(\varrho_p,u_p)$ to (\ref{lambda}), we can now pass to the limit with $\varepsilon\to 0$ and $p\to\infty$ in order to obtain (\ref{eq_limit}). For simplicity we set $a=1$ (then $\Lambda$ is explicitly given by $-\frac{\varepsilon}{b}\ln(1-b|z|)$ for $|z|<\frac{1}{b}$). Since we know that for the limit we have $|\ddiv u|\leq 1$, we also set $b<1$. This ensures us that 
\[ \int_0^t\int_{\T^d}\Lambda(|\ddiv u|)\;\dd x\dd s <\infty \]
for fixed $\varepsilon$.
    
    Let us take $(\varrho_p,u_p)$ being a solution to (\ref{lambda_variational}). By putting $\varphi=0$ in (\ref{lambda_variational}), since $\Lambda\geq 0$, we recover the energy estimate (\ref{energy}). Therefore there exists a limit $(\varrho,u)$ such that up to a subsequence $(\varrho_p,u_p)$ converges to it as in (\ref{convergence}). Note that from Propositions \ref{Du_bdd} and \ref{rho_bdd} we know that $|\D u|\leq 1$ and $\|\varrho(t,\cdot)\|_{L^\infty}\leq \|\varrho_0\|_{L^\infty}e^t$. In particular, this means that $u$ is an admissible test function for (\ref{lambda_variational}) and since $b<1$, 
    \[ \Lambda(|\ddiv u|) \leq \frac{\varepsilon }{b}|\ln(1-b)|<+\infty \quad \text{a. e.} \]
    Putting $\varphi=u$ and using the renormalized continuity equation and monotonicity of $\mathbb{S}_p$ we get
    \[\begin{aligned} \frac{1}{\gamma-1}\int_{\T^d}(\varrho_p^\gamma-\varrho^\gamma)(t,x)\;\dd x \leq& -\int_0^t\int_{\T^d}(\varrho_p^\gamma-\varrho^\gamma)\ddiv u\;\dd x\dd t -\int_0^t\int_{\T^d}\mathbb{S}_p(\D u):(\D u_p-\D u)\;\dd x\dd s \\
    &+ \int_0^t\int_{\T^d}\Lambda(|\ddiv u|)\;\dd x\dd s - \int_0^t\int_{\T^d}\Lambda(|\ddiv u_p|)\;\dd x\dd s. \end{aligned}\] 
    Finally, by the convexity of $\Lambda$ we have
    \[\begin{aligned} \int_0^t\int_{\T^d}\Lambda(|\ddiv u|)\;\dd x\dd s - \int_0^t\int_{\T^d}\Lambda(|\ddiv u_p|)\;\dd x\dd s &\leq \int_0^t\int_{\T^d}\Lambda'(|\ddiv u|)(\ddiv u-\ddiv u_p)\;\dd x\dd s \\
    &= \varepsilon\int_0^t\int_{\T^d}\frac{b}{1-b|\ddiv u|}(\ddiv u_p-\ddiv u)\;\dd x\dd s. \end{aligned}\]
    In the end, taking $\varepsilon=\frac{1}{p}$ we obtain the inequality
    \[ \frac{1}{\gamma-1}X_p(t) \leq \|\ddiv u\|_{L^\infty}\int_0^t X_p(s)\;\dd s + \varepsilon_p(t), \]
    where similarly as before
    \[ X_p(t) = \int_{\T^d}\varrho_p^\gamma-\varrho^\gamma-\gamma\varrho^{\gamma-1}(\varrho_p-\varrho)\;\dd x \geq 0 \]
    and
    \[\begin{aligned} \varepsilon_p(t) =& \gamma\int_{\T^d}\varrho^{\gamma-1}(\varrho_p-\varrho)\;\dd x -\gamma\int_0^t\int_{\T^d}\varrho^{\gamma-1}(\varrho_p-\varrho)\ddiv u\;\dd x\dd s \\
    &- \int_0^t\int_{\T^d}\mathbb{S}_p(\D u):(\D u_p-\D u)\;\dd x\dd s + \frac{1}{p}\int_0^t\int_{\T^d}\Lambda'(|\ddiv u|)(\ddiv u-\ddiv u_p)\;\dd x\dd s \end{aligned}\]
    which goes to zero a. e. as $p\to\infty$. Therefore using Gronwall's lemma similarly as before we have
    \[ \lim_{p\to\infty}X_p(t)=0 \quad \text{a. e.}, \]
    which implies the strong convergence of $\varrho$. This allows us to obtain the equation (\ref{eq_limit}) in the limit.
\end{proof}

\section{A Singular approach to get the same limit system}
Recently in \cite{BPZ} an unilateral constraint related to the maximal packing for a compressible Navier-Stokes  system has been obtained as a singular limit ($\varepsilon$ tends to infinity) of the usual compressible Navier-Stokes system
\begin{equation}
\left\{
\begin{aligned}
&\partial_t \varrho_\varepsilon + \partial_x (\varrho_\varepsilon u_\varepsilon) = 0,\\
&\partial_t (\varrho_\varepsilon u_\varepsilon) + \partial_x (\varrho_\varepsilon u_\varepsilon^2)
- \mu \partial_x^2 u_\varepsilon 
+   \varepsilon \partial_x \frac{\varrho_\varepsilon^\gamma}{(1-\varrho_\varepsilon)^\beta}
+ a \partial_x \varrho_\varepsilon^\gamma = 0
\end{aligned}
\right.\label{BrPeZa}
\end{equation}
with $a>0$, $\gamma>1$ and $\gamma,\beta >1$.
More precisely, when $\varepsilon$ tends to zero, the authors showed that for some topology (that we do not precise here):
$$\lim_{\varepsilon \to 0}
 \varrho_\varepsilon \to \varrho, \qquad \lim_{\varepsilon \to 0}
\varepsilon \frac{\varrho_\varepsilon^\gamma}  {(1-\varrho_\varepsilon)^\beta} = \pi$$ 
with 
$$0\le \varrho \le 1, \qquad 
   \pi\ge 0 \hbox{ with } \pi (1-\varrho)= 0$$
where $\pi$ is the Lagrangian multiplier associated to the maximal packing constraint $\varrho =1.$ The limit system
reads
\begin{equation}
\left\{
\begin{aligned}
& \partial_t\varrho + \partial_x(\varrho u) = 0,,\\
&\partial_x(\varrho u) + \partial_x(\varrho u^2) 
- \mu \partial_x^2 u + \partial_x \pi
+ a \partial_x \varrho^\gamma = 0.
\end{aligned}
\right.\label{BrPeZa2}
\end{equation}
Note that such singular limit has been extended to multi-dimensional case  in \cite{PerZat}. Let us now consider the following one-dimensional system
\begin{equation}
\left\{
\begin{aligned}
& \partial_t \varrho_\varepsilon 
+ \partial_x (\varrho_\varepsilon u_\varepsilon) = 0,\\
&\partial_t (\varrho_\varepsilon u_\varepsilon) + \partial_x (\varrho_\varepsilon u_\varepsilon^2)
- \varepsilon \partial_x 
\Bigl(
\frac{\partial_x u_\varepsilon}{(1- |\partial_x u_\varepsilon|^2)^{1/2}}
\Bigr) 
+ a\partial_x \varrho_\varepsilon^\gamma = 0.
\end{aligned}
\right.\label{Sing1}
\end{equation}
in a periodic setting in space with initial conditions
\begin{equation}
(\varrho_\varepsilon,u_\varepsilon)\vert_{t=0} 
= (\varrho_{\varepsilon,0},u_{\varepsilon,0}).
\label{Sing2}
\end{equation}
Additionally, the solution satisfies
\begin{equation}
\int_{\T} \varrho_\varepsilon dx = \int_{\T} \varrho_{\varepsilon, 0} dx = M_{\varepsilon,0},
\qquad
\int_{\T} \varrho_\varepsilon u_\varepsilon dx = 
\int_{\T} \varrho_{\varepsilon,0} u_{\varepsilon,0} dx
= \widetilde M_{\varepsilon,0}.
\label{Sing3} 
\end{equation}

Note that such system was also considered in the 3-dimensional case in \cite{FLM}.
We deduce from \cite{FLM} that
$|\partial_x u_{\varepsilon}| < 1$ almost everywhere. Note that the authors consider the initial density far from zero and bounded which remains true for all time.

\bigskip
In this setting, we are able to prove the following:
\begin{thm}\label{MainTh_singular} 
Let us consider $(\varrho_{\varepsilon,0},u_{\varepsilon,0})$ in 
$L^{\infty}(\T)\times W^{1,\infty}(\T)$ such that:
\[
\left\{
\begin{array}
[c]{l}%
\displaystyle 
\overline{E_{0}}:=\sup_{\varepsilon>0}\int_{\T}\left(  \frac{\varrho_{\varepsilon,0}u_{\varepsilon,0}^{2}%
}{2}+\frac{\varrho_{\varepsilon,0}^{\gamma}}{\gamma-1}\right)  \mathrm{d}x<\infty.\\
\\
\displaystyle 
0<c_{1}\leq\varrho_{\varepsilon,0}\leq c_{2}<+\infty 
\text{ and }\left\vert \partial_{x}u_{\varepsilon,0}\right\vert < 1,\\
\\
\displaystyle 
\varrho_{\varepsilon,0}\to\varrho_{0},\text{ }u_{\varepsilon,0}\to
u_{0}\text{ strongly in }L^{\infty}(\T)
\end{array}
\right.
\]
with $\gamma>1$ for some finite constants $c_{1},c_{2}>0$. Then, there exists a global
weak solution of \eqref{Sing1}--\eqref{Sing3}  such that the following energy inequality holds
\[
\sup_{t\in(0,T]}\int_{\T}\left[\frac{1}{2}\varrho_{\varepsilon}|u_{\varepsilon}|^{2}+\frac{1}%
{\gamma-1}\varrho_{\varepsilon}^{\gamma}\right](t)\mathrm{d}x+
\varepsilon
\int_{0}^{t}\int_{\T}
\frac{|\partial_x u_\varepsilon|^2}{\sqrt{1-|\partial_x u_\varepsilon|^2}}
\leq\overline
{E_{0}}.%
\]
with uniform constants with respect to $\varepsilon$. 
Moreover, we have
$$|\partial_x u_\varepsilon| < 1 \hbox{ almost everywhere }.$$
Note that we also have
$$\int_0^T\int_{\T} \varrho_\varepsilon |\dot u_\varepsilon|^2
\le C$$
with $\dot u_\varepsilon = \partial_t u_\varepsilon +
 u_\varepsilon\partial_x  u_\varepsilon$ and where $C$ is a constant independent on $\varepsilon.$
Moreover 
$$\varrho_\varepsilon \to \varrho \hbox{ in } {\mathcal C}([0,T]; L^r(\T)) \hbox{ for all } r<+\infty.$$
$$u_\varepsilon \rightharpoonup  u \hbox{ in } {\mathcal C}([0,T];L^2(\T)) \cap L^r(0,T;W^{1,r}(\T))
  \hbox{ for all } r<+\infty$$
  $$
  u_\varepsilon\to u \hbox{ in } L^2((0,T)\times (\T)
  $$
$$\partial_t u_p \rightharpoonup \partial_t u
   \hbox{ in } L^2((0,T)\times \T)$$
$$\varepsilon \partial_x u_\varepsilon / \sqrt{1-|\partial_x u_\varepsilon|^2}
\rightharpoonup \tau \hbox{ in } L^2((0,T)\times \T).
$$
where 
\begin{equation}
\left\{
\begin{array}
[c]{r}%
\partial_{t}\varrho +\partial_{x}\left(  \varrho u\right)  =0,\\
\partial_{t}\left(  \varrho u\right)  +\partial_{x}\left(  \varrho
u^{2}\right)  - \partial_{x}\tau + a \partial_{x}\varrho^{\gamma}=0, \\
|\partial_x u |\le  1 \qquad  \hbox{ and } \qquad  \tau = \pi \partial_x u \hbox{ with }
\pi \ge 0 
\hbox{ with } 
\pi (1-|\partial_x u |) = 0,  \\
\varrho\vert_{t=0} = \varrho_0, \qquad
\varrho u\vert_{t=0}= \varrho_0 u_0
\end{array}
\right.\label{limit-system_sing}
\end{equation}

\end{thm}

\noindent {\bf Sketch of proof.} The proof follows the same lines as for the power-law system considered in Section \ref{1d_section}, whereas some steps are simplified due to the bound on $|\partial_xu_\varepsilon|$. Therefore we will just provide the sketch of the proof and the main ideas.

\noindent {\it Energy estimate.}
The energy estimate for such system
reads
provides the following control
$$\int_{\T} \left[\frac{1}{2} \varrho_\varepsilon |u_\varepsilon|^2 + a \varrho_\varepsilon^\gamma\right](t) 
+
\varepsilon \int_0^t \int_{\T}
\frac{|\partial_x u_\varepsilon|^2}{(1-|\partial_x u_\varepsilon|^2)^{1/2}} 
\le \overline{E_0} .
$$

\bigskip

\noindent {\it Some important properties.}
Note that we already know that
$$|\partial_x u_\varepsilon| < 1 \hbox{ a.e. }$$
Note also  that
$$\varepsilon 
\frac{|\partial_x u_\varepsilon|^2}{(1-|\partial_x u_\varepsilon|^2)^{1/2}}
= \varepsilon 
\frac{1}{(1-|\partial_x u_\varepsilon|^2)^{1/2}}
- \varepsilon 
(1-|\partial_x u_\varepsilon|^2)^{1/2}
$$
and therefore we get the uniform bound
$$ \varepsilon \int_0^t\int_{\T} 
\frac{1}{(1-|\partial_x u_\varepsilon|^2)^{1/2}} \le C.$$

\bigskip

\noindent {\it Bound related to $\varrho_\varepsilon (\dot u_\varepsilon)^2$.}
As in the case for the power law, we would have to look for an estimate on $\varrho_\varepsilon (\dot u_\varepsilon)^2$ and also we would need a strong convergence on the density.
Let us recall that 
$$\varrho_\varepsilon \dot u_\varepsilon
- \partial_x \sigma_\varepsilon = 0$$
with $\displaystyle \sigma_\varepsilon = \frac{\partial_x u_\varepsilon}{(1-|\partial_x u_\varepsilon|^2)^{1/2}}
- a\varrho_\varepsilon^\gamma$
and therefore we need to estimate
$$- \int_{\T} (\partial_t u_\varepsilon +
u_\varepsilon \partial_x u_\varepsilon)\partial_x
\Bigl(\frac{\partial_x u_\varepsilon}{(1-|\partial_x u_\varepsilon|^2)^{1/2}}
\Bigr) 
$$
and then
$$\int_{\T} a (\partial_t u_\varepsilon +
u_\varepsilon \partial_x u_\varepsilon)\partial_x \varrho_\varepsilon^\gamma
$$ 
Note first that
$$-\int_{\T} \partial_t u_\varepsilon
\partial_x\Bigl(\frac{\partial_x u_\varepsilon}{(1-|\partial_x u_\varepsilon|^2)^{1/2}}
\Bigr)
= -\frac{d}{dt} \int_{\T} (1-|\partial_x u_\varepsilon|^2)^{1/2}
$$
and secondly after calculations that
$$- \int_{\T} 
u_\varepsilon \partial_x u_\varepsilon \partial_x
\Bigl(\frac{\partial_x u_\varepsilon}{(1-|\partial_x u_\varepsilon|^2)^{1/2}}
\Bigr)
= \int_{\T} \frac{\partial_x u_\varepsilon}{\sqrt{1-|\partial_x u_\varepsilon|^2}^{1/2}}.
$$
Indeed
$$- \int_{\T} 
u_\varepsilon \partial_x u_\varepsilon \partial_x
\Bigl(\frac{\partial_x u_\varepsilon}{(1-|\partial_x u_\varepsilon|^2)^{1/2}}
\Bigr)= \int_{\T} \frac{|\partial_x u_\varepsilon|^2 
\partial_x u_\varepsilon }{(1-|\partial_x u_\varepsilon|^2)^{1/2}}
+\frac{1}{2} \int_{\T} \frac{u_\varepsilon \partial_x |\partial_x u_\varepsilon|^2)}{(1-|\partial_x u_\varepsilon|^2)^{1/2}}.
$$
and therefore
$$- \int_{\T} 
u_\varepsilon \partial_x u_\varepsilon \partial_x
\Bigl(\frac{\partial_x u_\varepsilon}{(1-|\partial_x u_\varepsilon|^2)^{1/2}}
\Bigr)= 
\int_{\T} \frac{|\partial_x u_\varepsilon|^2 
\partial_x u_\varepsilon }{(1-|\partial_x u_\varepsilon|^2)^{1/2}}
+ \int_{\T} \partial_x u_\varepsilon (1-|\partial_x u_\varepsilon|^2)^{1/2}
$$
which gives the conclusion. Recall also that as for the $p$-system 
$$ \int_{\T}  (\partial_t u_\varepsilon +
u_\varepsilon \partial_x u_\varepsilon)\partial_x \varrho_\varepsilon^\gamma
= - \frac{d}{dt} \int_{\T} \varrho_\varepsilon^\gamma 
     \partial_x u_\varepsilon
     - \int_{\T} \gamma \varrho_\varepsilon^\gamma |\partial_x u_\varepsilon|^2.
$$
Using that $|\partial_x u_\varepsilon| < 1 {\it a.e.}$ 
and using the energy estimates and hypothesis on initial data,
we get that
$$\int_0^T\int_{\T} \varrho_\varepsilon |\dot u_\varepsilon|^2 
\le C < +\infty$$
uniformly with respect to $\varepsilon$. Note that using now this information,
we get the following bound $\varepsilon \partial_x u_\varepsilon/ \sqrt{1-|\partial_x u_\varepsilon|^2} \in L^2(0,T;L^\infty(\T))$
in a similar way than for the $p$- compressible limit that means starting with the relation $$\partial_x \sigma_\varepsilon =
\varrho_\varepsilon \dot u_\varepsilon$$
and integrating in space.

\bigskip

\noindent {\it Compactness on $\varrho_\varepsilon$.}
To deduce the compactness on $\varrho_\varepsilon$, we just have to
noticed that $s\mapsto s/\sqrt{1-|s|^2}$ is monotone. Therefore we can follow the same lines than for the compactness on $\varrho_p$ using the bounds already proved previously. Therefore we can prove the strong convergence of $\varrho_\varepsilon$ to $\varrho$ in ${\mathcal C}([0,T]; L^r(\T))$ for all $r<+\infty$.

\bigskip

\noindent {\it The unilateral constraint and the limit.}
The unilateral constraint should read again
$$|\partial_x u|\le 1, \qquad
   \pi \ge 0 \hbox{ with } \pi (1-|\partial_x u|)= 0$$
or at least
   $$|\partial_x u| = 1 \hbox{ on the support of } \pi$$
   where $\pi$ is the Lagrangian multiplier associated to the maximal strain constraint $|\partial_x u| =1.$  The important property is that we know that normally 
$|\partial_x u_\varepsilon|\le 1$ 
will be satisfied a.e.  due to the singular behavior and not only on the limit $\partial_x u$. Denoting $\tau$ the weak limit of $\tau_\varepsilon
= \varepsilon |\partial_x u_\varepsilon|/(1-|\partial_x u_\varepsilon|^2)^{1/2}$
and assuming more integrability than an $L^1$ bound on 

$\tau_\varepsilon$, we can do the following formal calculation
$$\int_0^T\int_{\T} |\tau|
\le \liminf_{\varepsilon\to 0}
\varepsilon \int_0^T\int_{\T} \frac{|\partial_x u_\varepsilon|}{(1-|\partial_x u_\varepsilon|^2)^{1/2}}
=  \liminf_{\varepsilon\to 0}
\varepsilon \int_0^T\int_{\T} \frac{|\partial_x u_\varepsilon|^2}{(1-|\partial_x u_\varepsilon|^2)^{1/2}}.
$$

Then denoting $\partial_x \tau_\varepsilon 
= a\partial_x \varrho_\varepsilon^\gamma + \varrho_\varepsilon \dot u_\varepsilon
=f_\varepsilon$
with $\tau_\varepsilon = 
\varepsilon \partial_x u_\varepsilon/(1-|\partial_x u_\varepsilon|^2)^{1/2}$, we write
$$
\varepsilon \int_0^T\int_{\T} \frac{|\partial_x u_\varepsilon|^2}{(1-|\partial_x u_\varepsilon|^2)^{1/2}}
\le - <f_\varepsilon, u_\varepsilon>_{L^2(0,T;H^{-1}(\T))
\times L^2(0,T; H^1(\T))}.$$
Then using the weak convergence on $u_\varepsilon$ in $L^2(0,T;H^1(\T))$ and showing the strong convergence
in $L^2(0,T; H^{-1}(\T))$ of $f_\varepsilon$ to $\varrho \dot u + a\varrho^\gamma$. we can get
$$\int_0^T\int_{\T} |\tau|
\le \int_0^T\int_{\T}
  \tau \partial_x u$$
and therefore we get the same conclusion as in the power-law problem.

\bigskip

\paragraph{Acknowledgments.}  The first author gratefully 
acknowledges the partial support by the Agence Nationale pour la Recherche grant ANR-23-CE40-0014-01 (ANR Bourgeons). This work also benefited of the support of the ANR under France 2030 bearing the reference ANR-23-EXMA-004 (Complexflows project). The second author acknowledges the partial support by the Agence Nationale pour la Recherche grant CRISIS (ANR-20-CE40-0020-01). The third author gratefully acknowledges the Mathematical Institute of Planet Earth (IMPT) in France for a two-years post-doctoral grant and the Polish National Science Centre grant no. 2022/45/N/ST1/03900 (Preludium) for partially supporting this work.

\begin{appendices}
\section{A certain continuity property for the transport equation.}
For the reader's convenience, we present below the proof of the lemma originally presented as the Lemma 2.12 in the ArXiv preprint \url{https://arxiv.org/abs/1907.09171} (for the published version see \cite{BrBu}).

\begin{lem}
\label{transport_lemma}Consider $\varrho\in L^{\infty}((0,T)
;L^{\gamma}(\mathbb{T}^{d})) \cap C([0,T];L_{weak}%
^{\gamma}(\mathbb{T}^{d}))\cap L^{\frac{p\gamma}{p-1}}((0,T)
\times\mathbb{T}^{d}))$ and $u\in L^{p}((0,T)
;W^{1,p}(\mathbb{T}^{d})) $ for some $p>1$, verifying the transport
equation
\[
\partial_{t}\varrho+\operatorname{div}\left(  \varrho u\right)  =0\text{ in
}\mathcal{D}^{\prime}\left(  \left(  0,T\right)  \times\mathbb{T}^{d}\right)
\]
along with the fact that%
\[
\lim_{t\rightarrow0}\int_{\T^d}\varrho\left(  t,x\right)  \psi\left(
x\right)  dx=\int_{\T^d}\varrho_{0}\left(  x\right)  \psi\left(
x\right)  dx\text{ for all }\psi\in C_{per}^{\infty}(\mathbb{R}^{d}).
\]
Then
\begin{equation}
\lim_{s\rightarrow0}\frac{1}{s}\int_{0}^{s}\int_{\T^d}\left(
\varrho^{\gamma}\left(  \tau,x\right)  -\varrho_{0}^{\gamma}\left(  x\right)
\right)  d\tau dx=0. \label{time_mean_continuity}%
\end{equation}
\bigskip
\end{lem}

\begin{proof} First of all, it is
classical to recover that $\varrho\in C(  [0,T);L^{q}(\mathbb{T}%
^{d}))$ with $q\in\lbrack1,\gamma)$ and that%
\begin{equation}
\lim_{t\rightarrow0}\varrho\left(  t,\cdot\right)  =\varrho_{0}\text{ in }L^{p}\text{ for
all }p\in\lbrack1,\gamma). \label{weak_time_cont}%
\end{equation}
This is of course not sufficient in order to prove $\left(
\text{\ref{time_mean_continuity}}\right)  $. Let us consider a spatial
approximation of the identity $\left(  \omega_{\varepsilon}\right)
_{\varepsilon>0}=\left(  \frac{1}{\varepsilon^{3}}\omega\left(  \frac{\cdot
}{\varepsilon}\right)  \right)  _{\varepsilon>0}$. We will denote by
\[
\varrho_{\varepsilon}\left(  t,x\right)  =\omega_{\varepsilon}\ast\varrho\left(
t,x\right)  .
\]
We have that%
\[
\lim_{\varepsilon\rightarrow0}\left\Vert \varrho-\varrho_{\varepsilon}\right\Vert
_{L^{\frac{p\gamma}{p-1}}((0,T)\times\T^d)}=0.
\]
Moreover, using \ref{weak_time_cont} for all $\varepsilon>0$ we have that%
\begin{equation}
\lim_{t\rightarrow0}\varrho_{\varepsilon}\left(  t,\cdot\right)  =\omega
_{\varepsilon}\ast\varrho_{0}\text{ in }L^{\gamma}. \label{weak_time_continuity}%
\end{equation}
For example,
\[
\left\Vert \varrho_{\varepsilon}\left(  t,\cdot\right)  -\omega_{\varepsilon}%
\ast\varrho_{0}\right\Vert _{L^{\gamma}}\leq\left\Vert \omega_{\varepsilon
}\right\Vert _{L^{p\left(  \eta\right)  }(\T^d)}\left\Vert
\varrho\left(  t,\cdot\right)  -\varrho_{0}\right\Vert _{L^{\gamma-\eta}%
(\T^d)}.
\]
Next, we apply $\omega_{\varepsilon}$ for the transport equation such as to
obtain
\begin{equation}
\partial_{t}\varrho_{\varepsilon}^{\gamma}+\operatorname{div}\left(
\varrho_{\varepsilon}^{\gamma}u\right)  +\left(  \gamma-1\right)  \varrho
_{\varepsilon}^{\gamma}\operatorname{div}u=\gamma\varrho_{\varepsilon}^{\gamma
-1}r_{\varepsilon}\text{ in }\mathcal{D}^{\prime}\left(  \left(  0,T\right)
\times\T^d\right)  \label{asterix}%
\end{equation}
with
\[
r_{\varepsilon}\rightarrow0\text{ in }L^{\frac{p\gamma}{\gamma+p-1}}\left(
\left(  0,T\right)  \times\T^d\right)  .
\]
An important property is that for all $\varepsilon>0$ and a.e. $t\in\left(
0,T\right)  $ it holds true that
\begin{align}
h_{\varepsilon}\left(  t\right)   &  =\int_{\T^d}\gamma
\varrho_{\varepsilon}^{\gamma-1}\left(  t\right)  r_{\varepsilon}\left(
t\right)  -\left(  \gamma-1\right)  \varrho_{\varepsilon}^{\gamma}\left(
t\right)  \operatorname{div}u\left(  t\right)  .\nonumber\\
&  \leq\left(  \gamma-1\right)  \int_{\T^d}\varrho_{\varepsilon
}^{\gamma}\left(  t\right)  \left\vert \operatorname{div}u\left(  t\right)
\right\vert +\gamma\int_{\T^d}\varrho_{\varepsilon}^{\gamma
-1}\left\vert r_{\varepsilon}\right\vert \nonumber\\
&  \leq C_{\gamma}\left\Vert \varrho\left(  t\right)  \right\Vert _{L^{\frac{p\gamma}{p-1}
}(\T^d)}^{\gamma}\left\Vert \nabla u\left(  t\right)  \right\Vert
_{L^{p}(\T^d)}:=h\left(  t\right)  \in L^{1}\left(  0,T\right)  .
\label{important_estimate}%
\end{align}
Integrating the $\left(  \text{\ref{asterix}}\right)  $ we end up with%
\[
\frac{d}{dt}\int_{\T^d}\varrho_{\varepsilon}^{\gamma}\left(
t,x\right)  dx=h_{\varepsilon}\left(  t\right)  \in L^{1}\left(  0,T\right)
.
\]
But using $\left(  \text{\ref{weak_time_continuity}}\right)  $ along with the
last relation we obtain that the application $t\rightarrow$ $\int
_{\T^d}\varrho_{\varepsilon}^{\gamma}\left(  t\right)  $ is absolutely
continious and we may write that%
\[
\int_{\T^d}\varrho_{\varepsilon}^{\gamma}\left(  t,x\right)
dx=\int_{\T^d}(\omega_{\varepsilon}\ast\varrho_{0})^{\gamma}\left(
x\right)  dx+\int_{0}^{t}h_{\varepsilon}\left(  \tau\right)  d\tau.
\]
From this and $\left(  \text{\ref{important_estimate}}\right)  $ we learn that%
\[
\left\vert \int_{\T^d}\varrho_{\varepsilon}^{\gamma}\left(  t,x\right)
dx-\int_{\T^d}(\omega_{\varepsilon}\ast\varrho_{0})^{\gamma}\left(
x\right)  dx\right\vert \leq\int_{0}^{t}h\left(  \tau\right)  d\tau.
\]
Now, we know that $h\left(  t\right)  \in L^{1}\left(  0,T\right)  $ and
consequently the application $t\rightarrow\int_{0}^{t}h\left(  \tau\right)
d\tau$ is absolutely continuous and
\[
\lim_{t\rightarrow0}\int_{0}^{t}h\left(  \tau\right)  d\tau=0.
\]
Let us fix $\eta>0$. Using the above we obtain the existence of a $t_{\eta}>0$
such that for all $t\in\left(  0,t_{\eta}\right)  $ \textit{and for all
}$\varepsilon>0$ one has%
\[
\left\vert \int_{\T^d}\varrho_{\varepsilon}^{\gamma}\left(  t,x\right)
dx-\int_{\T^d}(\omega_{\varepsilon}\ast\varrho_{0})^{\gamma}\left(
x\right)  dx\right\vert \leq\int_{0}^{t}h\left(  \tau\right)  d\tau\leq\eta.
\]
By the triangle inequality, we have that for all $\varepsilon>0$ and
$t\in\left(  0,t_{\eta}\right)  $%
\begin{equation}
\left\vert \frac{1}{t}\int_{0}^{t}\int_{\T^d}\varrho_{\varepsilon
}^{\gamma}\left(  \tau,x\right)  dxd\tau-\int_{\T^d}(\omega
_{\varepsilon}\ast\varrho_{0})^{\gamma}\left(  x\right)  dx\right\vert \leq\eta.
\label{almost_almost}%
\end{equation}
For $t$ fixed arbitrarly in $\left(  0,t_{\eta}\right)  $ we use the fact that%
\[
\lim_{\varepsilon\rightarrow0}\left\Vert \varrho-\varrho_{\varepsilon}\right\Vert
_{L^{\frac{p\gamma}{p-1}}((0,T)\times\T^d)}=0
\]
and we pass to the limit in $\left(  \text{\ref{almost_almost}}\right)  $ in
order to obtain that for all $t\in\left(  0,t_{\eta}\right)  $%
\[
\left\vert \frac{1}{t}\int_{0}^{t}\int_{\T^d}\varrho^{\gamma}\left(
\tau,x\right)  dxd\tau-\int_{\T^d}\varrho_{0}^{\gamma}\left(  x\right)
dx\right\vert \leq\eta.
\]
Since $\eta$ was fixed arbitrarly, the last property translates that
\[
\lim_{t\rightarrow0}\frac{1}{t}\int_{0}^{t}\int_{\T^d}\varrho^{\gamma
}\left(  \tau,x\right)  dxd\tau=\int_{\T^d}\varrho_{0}^{\gamma}\left(
x\right)  dx.
\]
This concludes the proof of Lemma \ref{transport_lemma}.
\end{proof}
\end{appendices}


\bibliographystyle{abbrv}
\bibliography{biblio}

\end{document}